\documentclass[12pt,a4paper]{article}
\usepackage[utf8]{inputenc}
\usepackage{amsmath, amssymb, amsthm, verbatim, hyperref}
\usepackage[dvipsnames]{xcolor}
\usepackage{ragged2e}
\usepackage{marginnote}
\usepackage{enumerate}
\usepackage{makecell}
\usepackage{longtable}
\usepackage{epsfig}
\usepackage{tikz}
\usepackage{ytableau}
\usepackage{hyperref}
\usepackage{caption}
\usepackage{subcaption}
\hypersetup{hidelinks}
\usepackage[left=2.50cm, right=2.50cm, top=2.00cm, bottom=2.00cm]{geometry}
\usepackage{graphicx,epstopdf}
\usetikzlibrary {arrows.meta}

\newcommand{\BSfinite}{{\mathcal{BS}_{\mathrm{fin}}}}
\newcommand{\BSlimit}{{\mathcal{BS}_{\mathrm{\infty}}}}

\theoremstyle{plain}

\newtheorem{innercustomthm}{Theorem}
\newenvironment{repthm}[1]
  {\renewcommand\theinnercustomthm{#1}\innercustomthm}
  {\endinnercustomthm}

\newtheorem{thm}{Theorem}[section]
\newtheorem{prop}[thm]{Proposition}
\newtheorem{cor}[thm]{Corollary}
\newtheorem{lemma}[thm]{Lemma}
\newtheorem{conjecture}[thm]{Conjecture}

\theoremstyle{definition}
\newtheorem{definition}[thm]{Definition}
\theoremstyle{remark}

\newtheorem{remark}[thm]{Remark}

\newtheorem{example}[thm]{Example}

\title{Bulgarian Solitaire: A new representation for depth generating functions}
\author{A.J. Harris, Son Nguyen} 
\date{\today}

\allowdisplaybreaks

\begin{document}
\ytableausetup{centertableaux}

\maketitle

\begin{abstract}
Bulgarian Solitaire is an interesting self-map on the set of integer partitions of a fixed number $n$. As a finite dynamical system, its long-term behavior is well-understood, having
recurrent orbits parametrized by necklaces of beads with two colors black $B$ and white $W$. However, the behavior of the transient elements within each orbit is much less understood. 

Recent work of Pham  considered the orbits
corresponding to a family of necklaces $P^\ell$ that are concatenations of $\ell$ copies of a fixed primitive necklace $P$. She proved striking limiting behavior as $\ell$ goes to infinity: the level statistic for the orbit, counting how many steps it takes a partition to reach the recurrent cycle, has a 
limiting distribution, whose generating function
$H_p(x)$
is rational. Pham also conjectured that $H_P(x), H_{P^*}(x)$ share the same denominator whenever $P^*$ is obtained from $P$ by
reading it backwards and swapping $B$ for $W$.

    Here we introduce a new representation of Bulgarian Solitaire that is convenient for the study of these generating functions. We then use it to prove two instances of Pham's conjecture, showing that
    $$H_{BWBWB \cdots WB}(x)=H_{WBWBW \cdots BW}(x)$$ and that $H_{BWWW\cdots W}(x),H_{WBBB\cdots B}(x)$ share the same denominator.
\end{abstract}

\tableofcontents

%--------------------------------------------------

\section{Introduction}
\label{sec:intro}

    The game of Bulgarian Solitaire (BS) was introduced sometime in the late 20th century, and popularized by Martin Gardner in 1983. The game itself is very simple. A player starts with $n$ cards divided into a number of piles in weakly descending order. Now  keep repeating the \textit{Bulgarian Solitaire move} $\beta$ defined as follows: take one card from each pile, form a new pile and put the piles in weakly descending order. The game ends when a configuration of cards is repeated.

    The game can naturally be represented by partitions and Young diagrams, where in each move, we remove the first column and reinsert it as a new row as shown in Figure \ref{fig:bs_move}.

    \begin{figure}[h!]
        \centering
        \includegraphics[width = 0.5\textwidth]{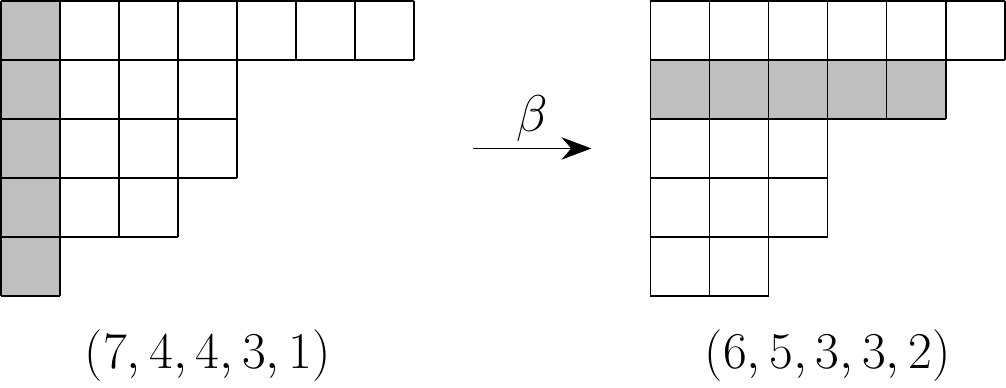}
        \caption{Bulgarian Solitaire move}
        \label{fig:bs_move}
    \end{figure}

    The BS move $\beta$ forms a dynamical system on the set $X$ of all partitions, and repeated application of $\beta$ leads to a \textit{recurrent cycle} $\mathcal{C}$ consisting of partitions $\lambda$ such that $\lambda = \beta^m(\lambda)$ for some $m$. Figure \ref{fig:bs_graph_8} shows an example of the Bulgarian Solitaire moves on partitions of $8$ in which the directed edges connect $\lambda$ to $\beta(\lambda)$. As can be seen in the example, we have two recurrent cycles $\{(3,2,2,1),(4,2,1,1),(4,3,1),(3,3,2)\}$ and $\{(3,3,1,1),(4,2,2)\}$.

    Brandt proved in \cite{brandt1982cycle} that there is a bijection between the set of recurrent cycles $\mathcal{C}$ and the set of objects called (black-white) necklaces. A {\it necklace} $N$ is an equivalence class of sequences of letters $\{B,W\}$ under cyclic rotation. Suppose $\binom{m}{2} \leq n < \binom{m+1}{2}$, then the bijection from the set of necklaces of length $m$ with $n-\binom{m}{2}$ $B$'s to the set of elements in the recurrent cycles for the BS system with $n$ cards is defined by
    \[ (b_1,b_2,\ldots,b_m) \rightarrow (m-1,m-2,\ldots,0) + (s_1,s_2,\ldots,s_m) \]
    where
    \begin{align*}
        s_i = \begin{cases}
            1 ~~ &\text{if } b_i = B \\
            0 ~~ &\text{if } b_i = W
        \end{cases}.
    \end{align*}

    \begin{figure}[h!]
        \centering
        \includegraphics[width = 0.9\textwidth]{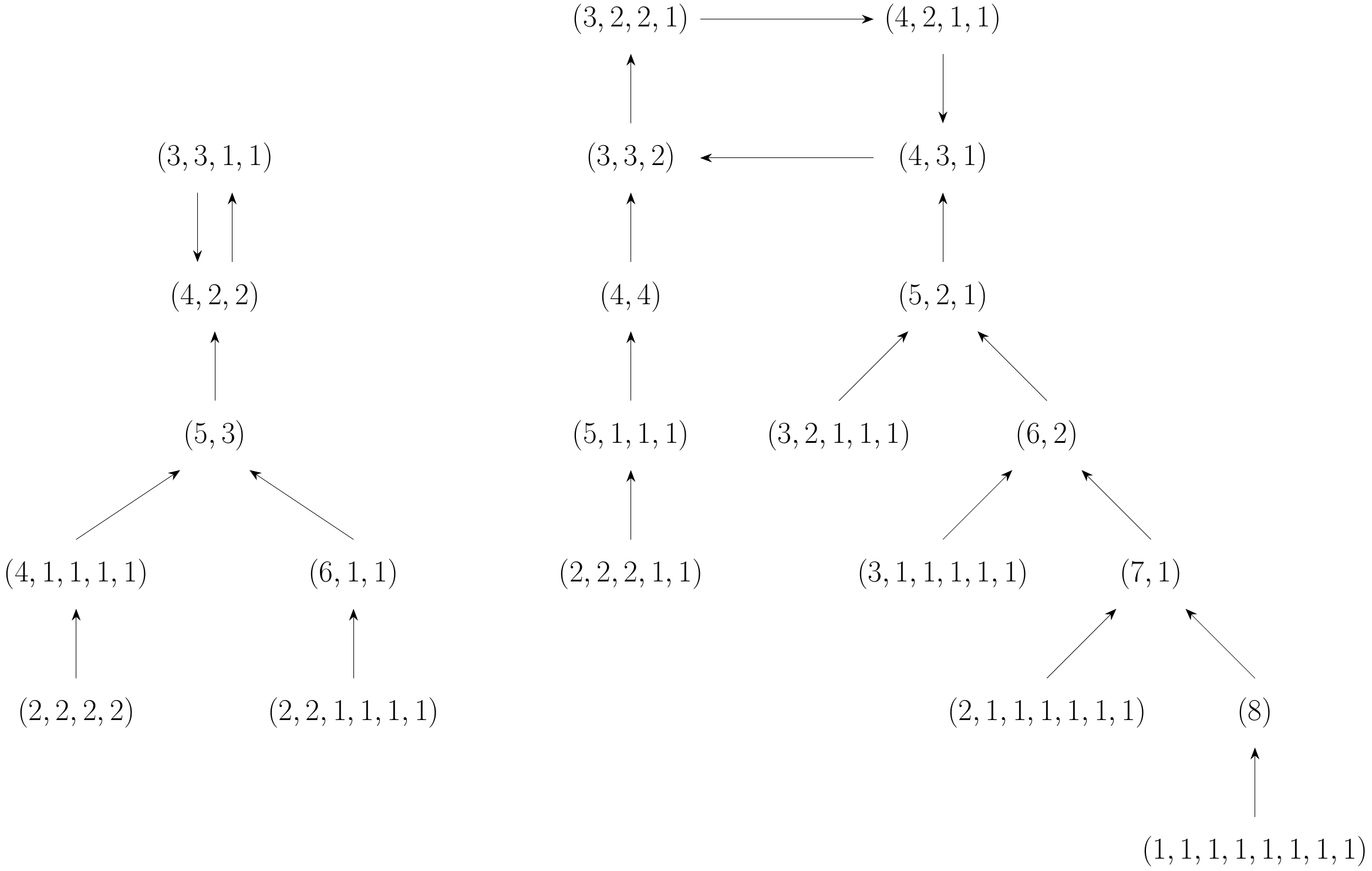}
        \caption{Example for $n = 8$}
        \label{fig:bs_graph_8}
    \end{figure}

    Figure \ref{fig:neck_cycle} gives a visualization of the bijection from necklaces of length $4$ with $2$ $B$'s to the recurrent elements of the system for $n = 8$.

    \begin{figure}[h!]
        \centering
        \includegraphics[width = 0.6\textwidth]{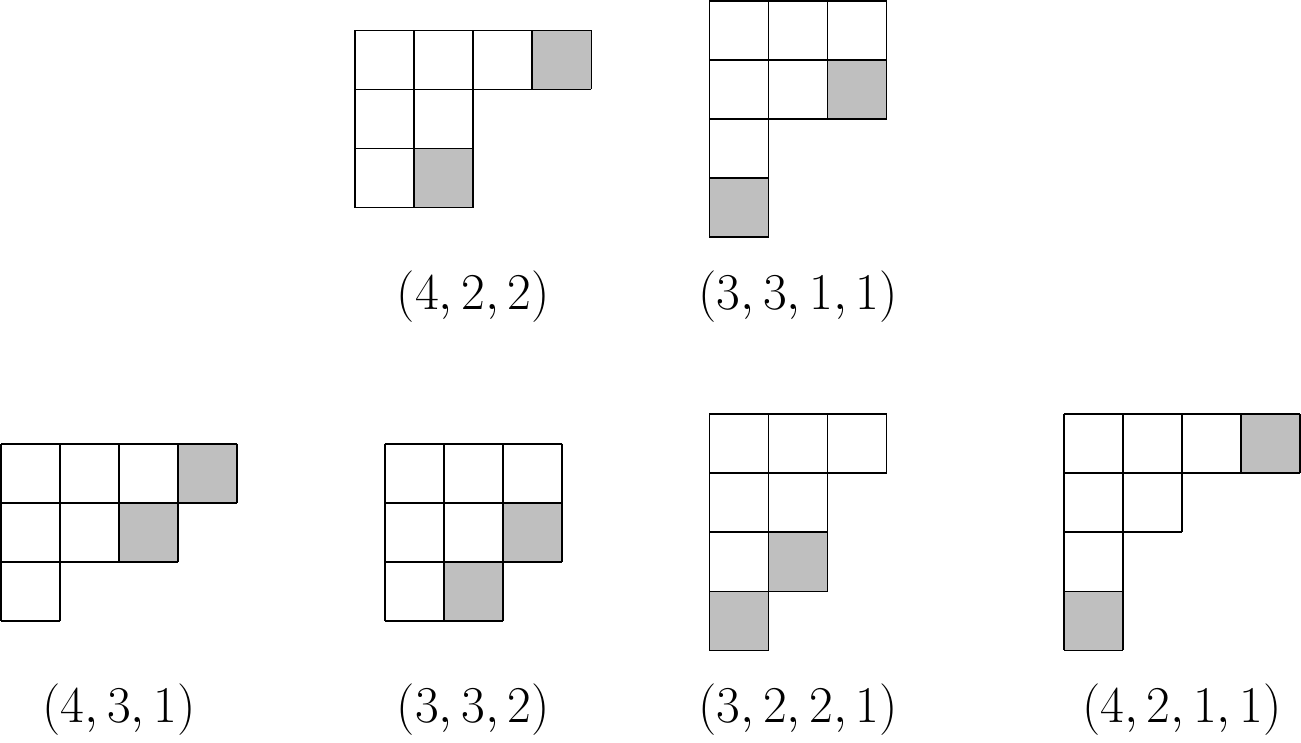}
        \caption{Necklaces and recurrent elements}
        \label{fig:neck_cycle}
    \end{figure}

    We call $P$ a \textit{primitive necklace} if it cannot be written as a concatenation $P = N^k = NN\ldots N$ with $k \geq 2$. For example, $BBWW$ is a primitive necklace while $BWBW = (BW)^2$ is not. For each necklace $N$, let $\mathcal{C}_N$ be the recurrent elements corresponding to necklaces in the equivalence class of $N$, and let the orbit $\mathcal{O}_N$ be the set of elements $\lambda$ such that $\beta^k(\lambda)\in\mathcal{C}_N$ for some $k \in \mathbb{Z}_{>0}$. For each element $\lambda$ in $\mathcal{O}_N$, let
    \[ \text{level}(\lambda) = \min\{k:\beta^k(\lambda)\in\mathcal{C}_N\} \]
    and define the level size generating function of $N$ to be
    \[ \mathcal{D}_N(x) = \sum_{\lambda\in\mathcal{O}_N} x^{\text{level}(\lambda)}. \]
    Our main results concern the limit of the generating function of $P^k$ as $k\rightarrow\infty$ for primitive necklaces $P$, that is
    \[ H_P(x) = \lim_{\ell\rightarrow\infty} \mathcal{D}_{P^\ell}(x). \]
    Eriksson and Jonsson proved in \cite[Section 4]{eriksson2017level} that such a limit exists when $P=W$, and
    \[ H_W(x) = \dfrac{(1-x)^2}{1-3x+x^2}. \]
    Pham then proved in \cite[Theorem 1.1, 1.2]{pham2022limiting} that for all primitive necklaces $P$ with $|P|\geq 2$, such a limit $H_P(x)$ exists. Furthermore, for $|P| \geq 3$, $H_P(x)$ is a rational function having denominator polynomial of degree at most $|P|$ and numerator degree at most $2|P|$.
    For example, she showed that
    \[ H_{BW}(x) = \dfrac{(1-x)^2(3x+2)}{x^3-3x^2-x + 1} \]
    and
    \[ H_{BWB}(x) = H_{WBW}(x) = \dfrac{(1-x)(x^3-3x^2-4x-3)}{2x^3+x^2-1}. \]
    Further computations led Pham to an interesting conjecture. For a primitive necklace $P$, the \textit{dual necklace} $P^*$ is constructed as follows: first reverse the order of the letters in $P$, then replace all letters $B$ by $W$ and vice versa. She made the following conjecture.

    \begin{conjecture}\label{con:P-and-P-dual}
        For all primitive necklaces $P$, the rational functions $H_P(x)$ and $H_{P^*}(x)$ 
        can be written with the same denominator of degree $|P|=|P^*|$.
    \end{conjecture}
    
    In this paper, we introduce a new representation of Bulgarian Solitaire and then use it to prove the following special cases of Conjecture \ref{con:P-and-P-dual}.

    \begin{thm}\label{thm:BWB_WBW}
        For $k\geq 1$, one has
        $H_{B(WB)^k}(x) = H_{W(BW)^k}(x)$ .
    \end{thm}

    \begin{thm}\label{thm:BWW_WBB}
        For $k\geq 1$, the functions $H_{BW^k}(x)$ and $H_{WB^k}(x)$ can both be written over the same denominator which is a polynomial of degree $k+1$.
    \end{thm}

    The paper is outlined as follows. In Section \ref{sec:setup}, we review some basic definitions and introduce our new representation. In Section \ref{sec:fuse}, we introduce fuses and pre-fuses, which will be important to our proof. Finally, we prove Theorem \ref{thm:BWB_WBW} in Section \ref{sec:thm-1.1} and Theorem \ref{thm:BWW_WBB} in Section \ref{sec:thm-1.2}.

\section{Set up}\label{sec:setup}

\subsection{Reversed Bulgarian Solitaire}\label{subsec:reversed_bs}

    It is actually more convenient to study the reversed Bulgarian Solitaire move rather than the (forward) Bulgarian Solitaire move. We give two analogous definition of the reversed BS move below.

    \begin{definition}[Reversed Bulgarian Solitaire move]\label{def:reversed_bs_move}
        For an element $\lambda$, a reversed Bulgarian Solitaire move $R_j$ maps $\lambda$ to $R_j(\lambda)$ as follows

        \begin{itemize}
            \item For Young diagrams: take out the $j$th row and insert it as the leftmost column.
            \item For a partition: take out the $j$th part and distribute it into the other parts, one for each.
        \end{itemize}
    \end{definition}

    For our move to make sense, the $j$th part needs to be as least $\ell(\lambda)-1$. Thus, the reversed BS moves are only defined for such parts. If $R_j$ is defined, we say the $j$th part is \textit{playable}. Figure \ref{fig:reversed_bs_ex} shows an example of $\lambda = (5,3,3,2)$, $R_1(\lambda)$ and $R_3(\lambda)$. Note that in the example, $R_2(\lambda)$ is also defined; however, since $\lambda_2 = \lambda_3$, $R_2(\lambda)$ and $R_3(\lambda)$ are the same. In general, for our convenience in later sections, if $\lambda_i = \lambda_{i+1} = \ldots = \lambda_j$, we will only consider $R_j(\lambda)$. Finally, in the example, $R_4(\lambda)$ is not defined since $\lambda_4 = 2 < 3 = \ell(\lambda) - 1$.

    \begin{figure}[h!]
        \centering
        \includegraphics[width = 0.4\textwidth]{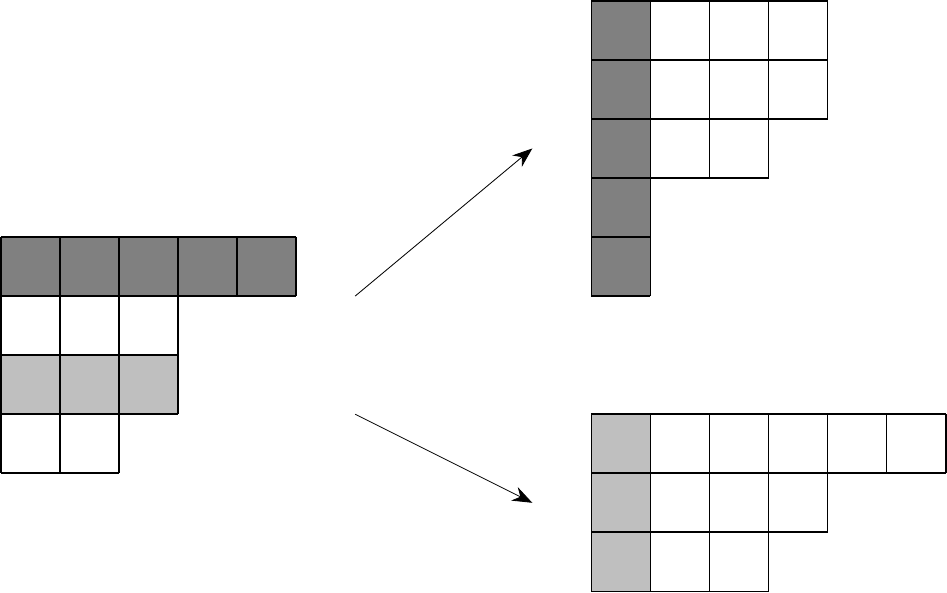}
        \caption{Reversed BS moves}
        \label{fig:reversed_bs_ex}
    \end{figure}

\subsection{New representation}\label{subsec:new_rep}

    Now we define our new representation of Bulgarian Solitaire.

    \begin{definition}\label{def:new_rep}
        Given a partition $\lambda=(\lambda_1 \geq \lambda_2 \geq \cdots \geq \lambda_\ell)$,
        we will instead view it as an infinite sequence $\lambda=(\lambda_1,\lambda_2,\ldots,\lambda_\ell,0,0,\ldots)$ of nonnegative integers that is eventually zero.
        For such a partition $\lambda$, define $\mu_{\lambda}= (\mu_1,\mu_2,\ldots)$ by 
        \[ \mu_i = \lambda_i - \lambda_{i+1}.\]
        In addition, if the $j$th part is playable and $\mu_j \neq 0$, we put a bar above $\mu_j$.
        Note that $\mu$ is also an infinite sequence of nonnegative integers which is eventually zero.
        Call this set of sequences $\BSfinite$.
    \end{definition}

    For example, for $\lambda = (5,3,3,2,0,0,\ldots)$, we have $\mu_\lambda = (\overline{2},0,\overline{1},2,0,0,\ldots)$. When the context is clear, we may omit the subscript $\lambda$. Observe that we can easily recover $\lambda$ from $\mu_\lambda$ by the following formula:
    \[ \lambda_i = \sum_{k = i}^\infty \mu_i. \]
    It is also not difficult to determine which part of $\mu$ is playable.
    
    \begin{lemma}\label{lem:valid-mu}
        Consider a sequence $\mu\in\BSfinite$, let $\ell$ be the largest index such that $\mu_{\ell} > 0$. There is a bar above $\mu_i$ if and only if $\sum_{k = i}^{\ell} \mu_i \geq \ell-1$.
    \end{lemma}

    \begin{proof}
        This follows directly from the fact that the $j$th part of a partition $\lambda$ is playable if and only if its size is at least $\ell(\lambda) - 1$.
    \end{proof}

    From now on, we will refer to the elements in $\BSfinite$ by their new representation $\mu:=\mu_\lambda$ instead of the standard partition representation $\lambda$. Furthermore, when we refer to an element $\mu$ in $\BSfinite$, we assume that the bars in $\mu$ satisfy the conditions in Lemma \ref{lem:valid-mu}.

    The following lemma shows that this new representation behaves nicely under the reversed BS moves.

    \begin{lemma}\label{lem:bs_move_result}
        For any $\lambda$ such that the $j$th part is playable, let $\lambda' = R_j(\lambda)$, $\mu = \mu_{\lambda}$ and $\mu' = \mu_{\lambda'}$. Then the parts $\mu_i'$ are determined by the parts of $\mu$ in these three cases:

        \begin{itemize}

            \item[(1)] If $j = 1$ then
                \begin{align*}
                    \mu'_i = \begin{cases}
                        \mu_{i+1} ~~ &\text{if }  i \neq \lambda_1\\
                        \mu_{i+1}+1 ~~ &\text{if }  i = \lambda_1
                    \end{cases}
                \end{align*}

            \item[(2)] If $j\geq 2$ and $\lambda_j \neq j-1$
                \begin{align*}
                    \mu'_i = \begin{cases}
                        \mu_i ~~ &\text{if } i<j-1 \\
                        \mu_{i-1}+\mu_i ~~ &\text{if } i = j-1 \\
                        \mu_{i+1} ~~ &\text{if } i\geq j~\text{and}~ i \neq \lambda_j\\
                        \mu_{i+1}+1 ~~ &\text{if } i\geq j~\text{and}~ i = \lambda_j
                    \end{cases}
                \end{align*}

            \item[(3)] If $j\geq 2$ and 
            $\lambda_j = j-1$
                \begin{align*}
                    \mu'_i = \begin{cases}
                        \mu_i ~~ &\text{if } i<j-1 \\
                        \mu_{i-1}+\mu_i + 1 ~~ &\text{if } i = j-1 \\
                        \mu_{i+1} ~~ &\text{if } i\geq j
                    \end{cases}
                \end{align*}  
        \end{itemize}
    The bars on the parts of $\mu'$ are determined as follows. For $i \leq j-1$, put a bar above $\mu'_i$ if $\mu'_i\neq 0$. For $i\geq j$, put a bar above $\mu'_i$ if $\mu'_i \neq 0$ and $\sum_{k = j}^{i} \mu_k<3$.
    \end{lemma}
       When the context is clear, we will denote $\mu' = R_j(\mu)$.

    \begin{proof}
        The three cases follow the same reasoning: we distribute 1 to each part, so the difference between 2 consecutive parts stay the same except between $\lambda'_{j-1}$ and $\lambda'_j$. Since $\lambda'_j = \lambda_{j+1}+1$, we have $\lambda'_{j-1} - \lambda'_j = \lambda_{j-1} - \lambda_{j+1} = \mu_{j-1}+\mu_j$. However, observe that we add $1$ to part $\lambda_j$ but none to part $\lambda_{j}+1$, so the difference is increased by $1$. Finally, to determine the bars, observe that $\lambda_j$ is the length of $\lambda'$. Thus, for the $i$th part of $\lambda'$ to be playable, $\lambda'_i \geq \lambda_j - 1$, which means $\lambda_i \geq \lambda_j - 2$. This is obviously true for $i<j$, and is equivalent to the condition $\sum_{k = j}^{i} \mu_i<3$ for $i\geq j$.
    \end{proof}

    For example, in Figure \ref{fig:reversed_bs_ex}, $\mu = (\overline{2},0,\overline{1},2,0,0,\ldots)$, so $R_1(\mu) = (0,\overline{1},2,0,1,0,0,\ldots)$ and $R_3(\mu) = (\overline{2},\overline{1},\overline{3},0,0,\ldots)$.

\subsection{The system in the limit}\label{subsec:BSlimit}

    Let us now shift to our main concern of the paper, the {\it limiting version} of the Bulgarian Solitaire system. We will start with an example with the primitive necklace $P=BWW$ and its powers $P^1,P^2,P^3,\cdots$. 
    
    \begin{figure}[h!]
    \centering
        \begin{minipage}{.4\textwidth}
          \centering
            \includegraphics[scale = 0.3]{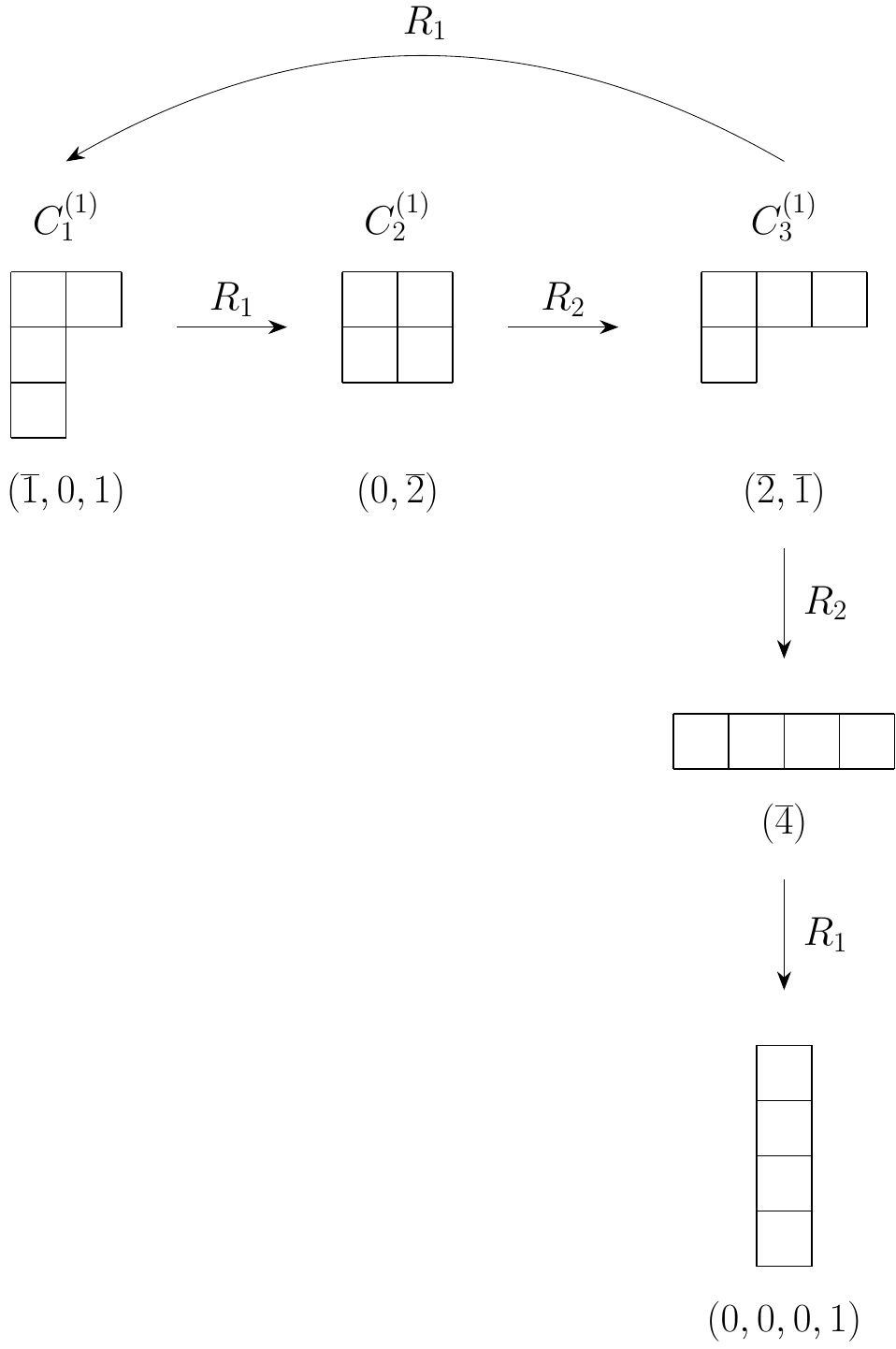}
            \caption{Digraph $\mathcal{O}_{(BWW)^1}$}
            \label{fig:forest_1}
        \end{minipage}%
        \begin{minipage}{.6\textwidth}
          \centering
            \includegraphics[scale = 0.3]{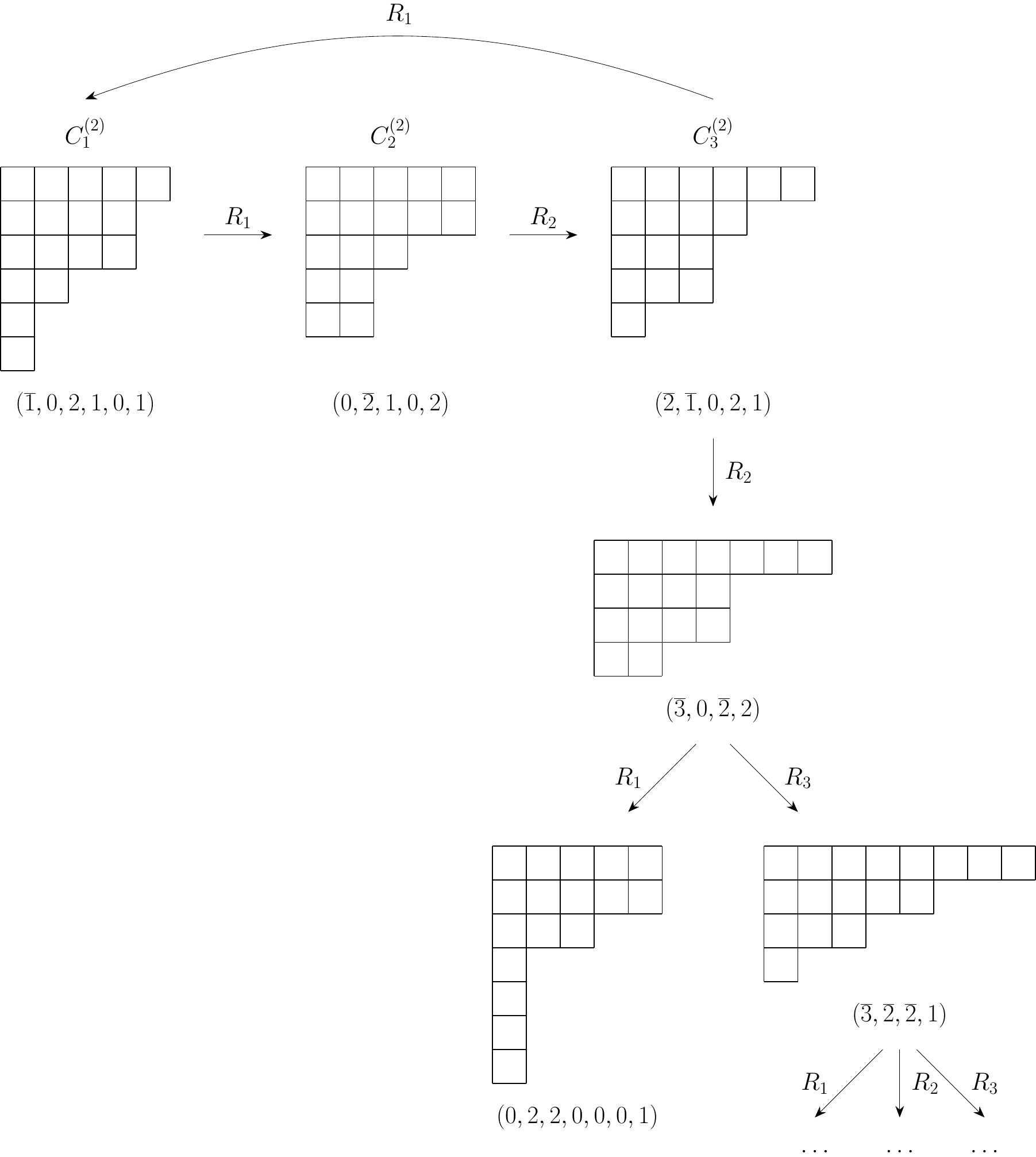}
            \caption{Part of the digraph $\mathcal{O}_{(BWW)^2}$}
            \label{fig:forest_2}
        \end{minipage}
    \end{figure}
    
    \begin{figure}[h!]
        \centering
        \includegraphics[scale = 0.3]{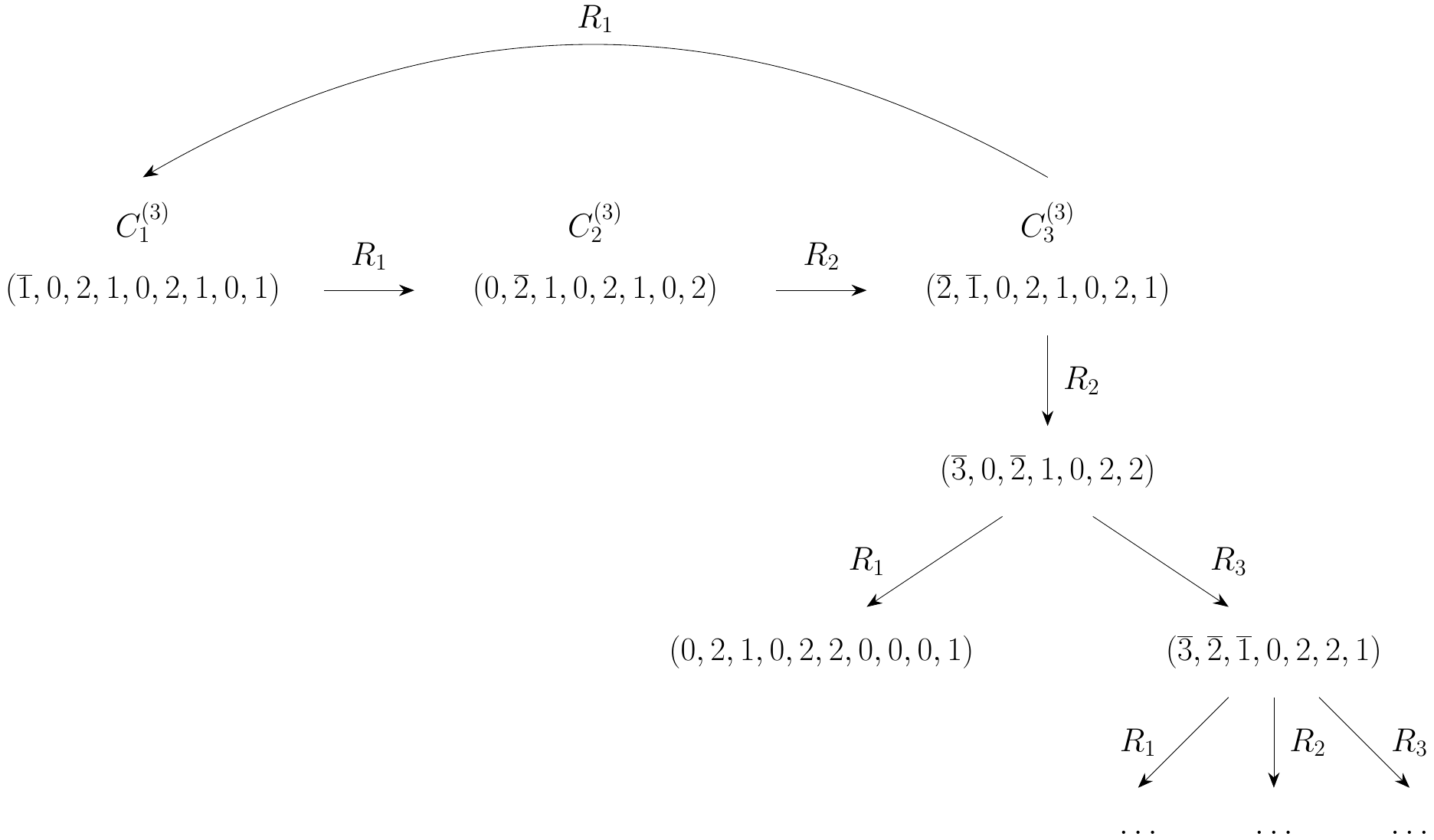}
        \caption{Part of the digraph $\mathcal{O}_{(BWW)^3}$}
        \label{fig:forest_3}
    \end{figure}
    
    Figures \ref{fig:forest_1}, \ref{fig:forest_2}, and \ref{fig:forest_3} show the digraphs for reverse BS orbits $\mathcal{O}_{(BWW)^1},\mathcal{O}_{(BWW)^2}, \mathcal{O}_{(BWW)^3}$. Observe that their recurrent sets 
    $$
    \begin{aligned}
    %\mathcal{O}_{(BWW)^1}|_C&=
    \{C^{(1)}_1,C^{(1)}_2,C^{(1)}_3\}&\leftrightarrow \{BWW, \quad WBW, \quad WWB\},\\
    %\mathcal{O}_{(BWW)^2}|_C&=
    \{C^{(2)}_1,C^{(2)}_2,C^{(2)}_3\} &\leftrightarrow \{BWWBWW,\quad WBWWBW,\quad WWBWWB\} \\
    %\mathcal{O}_{(BWW)^3}|_C&=
    \{C^{(3)}_1,C^{(3)}_2,C^{(3)}_3\}&\leftrightarrow \{BWWBWWBWW, \quad WBWWBWWBW,\quad WWBWWBWWB\}\\
    \vdots
    \end{aligned}
    $$
    are in bijection, corresponding to the words
    within a cyclic equivalence class of necklaces of the form $P^\ell$
    for $\ell=1,2,3$. However, note that these bijections appear to extend to
    natural inclusions of digraphs 
    $$
    \mathcal{O}_{(BWW)^1} \hookrightarrow \mathcal{O}_{(BWW)^2} \hookrightarrow \mathcal{O}_{(BWW)^3} \hookrightarrow \cdots.
    $$
    Furthermore, the first three levels of $\mathcal{O}_{(BWW)^2}$ and $\mathcal{O}_{(BWW)^3}$ are isomorphic.
    
    Indeed, Pham \cite{pham2022limiting} showed that for any $m$, one can find $L(m)$ large enough such that the first $m$ levels of $\mathcal{O}_{(BWW)^\ell}$ are isomorphic for all $\ell > L(m)$. In other words, $\mathcal{O}_{(BWW)^\ell}$ converges to a digraph $\mathcal{O}_{(BWW)^\infty} = \lim_{\ell\rightarrow\infty}\mathcal{O}_{(BWW)^\ell}$. The recurrent set $\{C_1,C_2,C_3\}$ of $\mathcal{O}_{(BWW)^\infty}$ is $\{(BWW)^\infty, (WBW)^\infty, (WWB)^\infty\}$, which is still in bijection with $\{BWW,WBW,BWB\}$. Observe that $C_1$, $C_2$, and $C_3$ are infinite sequences with period $3$.
     
    Pham generalized this idea to every primitive necklace $P$ of length $|P|=n$. 
    She showed in \cite{pham2022limiting} that for the game digraphs $\mathcal{O}_{P^\ell}$, 
    the bijections of their recurrent sets $\{C^{(\ell)}_1,C^{(\ell)}_2,\ldots,C^{(\ell)}_n\}$
    extend to digraph inclusions
    $$
     \mathcal{O}_{P} \hookrightarrow \mathcal{O}_{P^2} \hookrightarrow \mathcal{O}_{P^3} \hookrightarrow \cdots
    $$
    which converge to a digraph $\mathcal{O}_{P^\infty}$, with recurrent set $\{C_1,C_2,\ldots,C_n\}$. Each $C_i$ in the recurrent set is an infinite sequence with period $n$. 

    Furthermore, one can write out the $C_i$'s in the recurrent set $\{C_1,C_2,\ldots,C_n\}$ in terms of the new representation as follows. Let $(b_1,b_2,\ldots,b_n)$ be a word in the cyclic equivalence class of the necklace $P$ (where $|P| = n$), then the corresponding $C_i$ in the recurrent set of $\mathcal{O}_{P^\infty}$ is $(\mu_1,\mu_2,\ldots)$ where
    \begin{align}\label{eqn:mu-from-b}
        \mu_i = \begin{cases}
            2 ~~ &\text{if } b_ib_{i+1} = BW \\
            1 ~~ &\text{if } b_ib_{i+1} = BB~\text{or}~WW \\
            0 ~~ &\text{if } b_ib_{i+1} = WB
        \end{cases}.
    \end{align}
    Here the indices are taken mod $n$. For example, the recurrent set of $\mathcal{O}_{(BWW)^\infty}$ is
    \[ \{(\overline{2},\overline{1},0,2,1,0,\ldots),\quad (0,\overline{2},1,0,2,1,\ldots),\quad (\overline{1},0,2,1,0,2,\ldots)\}. \]
    The following properties are straightforward from (\ref{eqn:mu-from-b}).

    \begin{lemma}\label{lem:tail-P}
        Let $\mu$ be an element in the recurrent set of $\mathcal{O}_{P^\infty}$ where $|P| = n$, then we have
        \begin{itemize}
            \item $\mu_i \in \{0,1,2\}$ for all $i\geq 1$; furthermore, the $0$'s and $2$'s alternate;
            \item $\mu_i = \mu_{i+n}$ for all $i\geq 1$; and
            \item $\mu_{i}+\mu_{i+1}+\ldots+\mu_{i+n-1} = n$ for all $i\geq 1$.
        \end{itemize}
        We also say that this sequence has period $n$.
    \end{lemma}

    \begin{proof}
        We get $\mu_i = \mu_{i+n}$ for all $i\geq 1$ from (\ref{eqn:mu-from-b}). Now, we prove that the $0$'s and $2$'s alternate. A $2$ only appears in the sequence if in the necklace we go from $B$ to $W$. Then the next entries will all be $1$ until in the necklace we go back from $W$ to $B$. When we go from $W$ to $B$, the corresponding is $0$, so we cannot have two consecutive $2$'s. The same argument shows that we cannot have two consecutive $0$'s. This also implies the third condition. Since $\mu_i \in \{0,1,2\}$ for all $i\geq 1$, and the $0$'s and $2$'s alternate, among every $n$ consecutive entries, there are as many $0$'s as $2$'s, so the sum is $n$.
    \end{proof}

    \begin{definition}\label{def:tail-P}
        We call a (possibly finite) sequence of integers $(a_1,a_2,\ldots)$ a \textit{proper tail of period $n$} if it satisfies the conditions in Lemma \ref{lem:tail-P} with period $n$. If the sequence is finite, we require the $0$'s and $2$'s to alternate when the sequence is read cyclically.
    \end{definition}

    For example, the sequence $(2,1,0,2,1,0,\ldots)$ is a proper tail of period $3$. However, the sequence $(2,1,0,2)$ is not a proper tail because if we read the sequence cyclically, we obtain the sequence $(2,1,0,2,2,1,0,2,\ldots)$ in which the $0$'s and $2$'s do not alternate. On the other hand, the sequence $(2,1,0,2,1,0)$ is a proper tail.

    \begin{lemma}\label{lem:tail-biject-P}
        The set of necklaces of length $n$ bijects with the set of proper tails of period $n$ with length $n$.
    \end{lemma}

    \begin{proof}
        The bijection is the map defined by (\ref{eqn:mu-from-b}). Clearly, the map is injective. Constructing the inverse is also simple. For a proper tail $(a_1,\ldots,a_n)$, let $i$ be the smallest index such that $a_i = 2$. We construct the necklace $(b_1,\ldots,b_n)$ by first setting $b_i = B$ and $b_{i+1} = W$. Next, we iterate from $j:=i+1$ to $n$; then, we iterate from $j:=1$ to $i-1$. In each iteration, if $a_j = 0$ then set $b_j = B$, and if $a_j = 2$ then set $b_j = W$. If $a_j = 1$ then set $b_{j} = b_{j-1}$ if $j\neq 1$ and $b_j = b_n$ is $j = 1$. The condition that the $0$'s and $2$'s alternate when the sequence is read cyclically assures that we get the correct inverse.
    \end{proof}

    Therefore, from now, we can associate proper tails with necklaces.

    \begin{definition}
        Let $a = (a_1,a_2,\ldots)$ be a proper tail of period $n$. We say $a$ is a \textit{proper tail of $P$}, where $P$ is a primitive necklace of length $n$, if the bijection defined by (\ref{eqn:mu-from-b}) maps $(a_1,\ldots,a_n)$ to a necklace in the cyclic equivalence class of $P$.
    \end{definition}

    Now, we characterize the sequences of nonnegative integers $\mu=(\mu_1,\mu_2,\ldots)$ that can occur in the limit of the Bulgarian digraphs $\mathcal{O}_{P^\infty}$, along with the possible positions of bars $\overline{\mu}_j$ indicating that a reversed BS move $R_j$ in position $j$ is applicable.

    \begin{definition}\label{def:O_P_infty}
        For a primitive necklace $P$ with $|P| = n$, we define $\mathcal{O}'_P$ to be the set of all $\nu$ that can be constructed as follows.
        \begin{enumerate}
            \item Pick $\mu \in \mathcal{O}_{P^\ell}$ for some $\ell$ such that there is an index $i$ satisfying
            \begin{itemize}
                \item $(\mu_i,\mu_{i+1},\ldots,\mu_{i+n-1})$ is a proper tail of $P$, and
                \item none of the entries $\mu_i,\mu_{i+1},\ldots,\mu_{i+n-1}$ are barred, i.e. none of the positions $i,i+1,\ldots,i+n-1$ are playable.
            \end{itemize}
            \item Replace $(\mu_{i+n},\mu_{i+n+1},\ldots)$ with infinitely many copies of $(\mu_i,\mu_{i+1},\ldots,\mu_{i+n-1})$.
        \end{enumerate}
    \end{definition}

    \begin{prop}\label{prop:O_P_infty_elem}
        For any primitive necklace $P$, the elements $\nu \in \mathcal{O}'_P$ are exactly the limits $\nu = \lim_{j\rightarrow\infty}\mu^{(\ell+j)}$ of convergent sequences $(\mu^{(\ell+1)},\mu^{(\ell+2)},\ldots)$ with $\mu^{(\ell+j)}\in \mathcal{O}_{P^{\ell+j}}$ for all $j$. In other words, $\mathcal{O}'_P = \mathcal{O}_{P^\infty}$.
    \end{prop}

    \begin{proof}
        Let $\mu^{(\ell)} \in \mathcal{O}_{P^\ell}$ for some $\ell$, and $i$ is an index such that $(\mu_i,\mu_{i+1},\ldots,\mu_{i+n-1})$ is a proper tail of $P$. Moreover, none of the entries $\mu_i,\mu_{i+1},\ldots,\mu_{i+n-1}$ are barred. Then $\mu^{(\ell+j)}$ can be obtained from $\mu^{(\ell)}$ by adding $j$ copies of $(\mu_i,\mu_{i+1},\ldots,\mu_{i+n-1})$ between $\mu_{i+n-1}$ and $\mu_{i+n}$. Furthermore, we claim that the bars of $\mu^{(\ell+j)}$ are in the exact same places as the bars in $\mu^{(\ell)}$. This is because the bars in $\mu^{(\ell)}$ can only possibly be on the first $i-1$ entries $\mu_1,\ldots,\mu_{i-1}$. By adding $j$ copies of $(\mu_i,\mu_{i+1},\ldots,\mu_{i+n-1})$ between $\mu_{i+n-1}$ and $\mu_{i+n}$, in $\mu^{(\ell+j)}$, there are $nj$ more parts than in $\mu^{(\ell)}$. In addition, since $\mu_i + \ldots + \mu_{i+n-1} = n$, for any $1\leq k \leq i-1$, $\sum_{r=k}^\infty \mu_{r}^{(\ell+j)} = \sum_{r=k}^\infty \mu_{r}^{(\ell)} + nj$. That is, for any $1\leq k \leq i-1$, the sum $\sum_{r=k}^\infty \mu_{r}^{(\ell+j)}$ is also exactly $nj$ more than $\sum_{r=k}^\infty \mu_{r}^{(\ell)}$. Thus, the $k$th part ($1\leq k\leq i-1$) is playable in $\mu^{(\ell+j)}$ if and only if it is playable in $\mu^{(\ell)}$.
        
        From this, $\lim_{j\rightarrow\infty}\mu^{(\ell+j)}$ can be obtained from $\mu^{(\ell)}$ by replace $(\mu_{i+n},\mu_{i+n+1},\ldots)$ with infinitely many copies of $(\mu_i,\mu_{i+1},\ldots,\mu_{i+n-1})$. This gives the corresponding element in $\mathcal{O}'_{P}$.
    \end{proof}

    This allows us to define the limit version of the Bulgarian Solitaire system.

    \begin{definition}\label{def:BSlimit}
        We define $\BSlimit$ to be the set of all elements obtained by the construction in Definition \ref{def:O_P_infty} for all primitive necklaces $P$.
    \end{definition}

    We can define the reversed BS moves on $\BSlimit$ similar to Lemma \ref{lem:bs_move_result}.

    \begin{lemma}
        \label{lem:bs_move_lim}
        If the $j$th part of $\mu$ is playable, i.e. there is a bar above $\mu_j$, we define $\mu' := R_j(\mu)$ as follows.
        
        \begin{itemize}

            \item[(1)] If $j = 1$ then
                \begin{align*}
                    \mu'_i = \mu_{i+1}
                \end{align*}

            \item[(2)] If $j\geq 2$ then
                \begin{align*}
                    \mu'_i = \begin{cases}
                        \mu_i ~~ &\text{if } i<j-1 \\
                        \mu_{i-1}+\mu_i ~~ &\text{if } i = j-1 \\
                        \mu_{i+1} ~~ &\text{if } i\geq j
                    \end{cases}
                \end{align*}
        \end{itemize}
    The bars on the parts of $\mu'$ are determined as follows. For $i \leq j-1$, put a bar above $\mu'_i$ if $\mu'_i\neq 0$. For $i\geq j$, put a bar above $\mu'_i$ if $\mu'_i \neq 0$ and $\sum_{k = j}^{i} \mu_k<3$.
        
    \end{lemma}

    \begin{proof}
        This construction is consistent with the one in Lemma \ref{lem:bs_move_result}. The only difference is that we exclude any case that includes $\lambda_j$. Recall that in Lemma \ref{lem:bs_move_result}, we need these cases because when playing $R_j$, we add $1$ to $\lambda_j$ but none to $\lambda_{j+1}$. In $\BSlimit$, $\lambda_j$ is $\infty$, so this situation does not arise, and we can exclude these cases.
    \end{proof}

    From now on, unless stated otherwise, when we refer to ``an element $\mu$'', we mean that $\mu$ is in $\BSlimit$. Finally, recall that we can write out the recurrent cycle elements in $\mathcal{O}_{P^\infty}$ (in $\BSlimit$). It is a bit trickier to decide which part is playable, i.e. where to put the bars. One strategy is to play the first nonzero part of one element to see which part of the next element is playable. For example, the recurrent set of $\mathcal{O}_{(BWW)^\infty}$ is
    \[ \{(\overline{2},\overline{1},0,2,1,0,\ldots),\quad (0,\overline{2},1,0,2,1,\ldots),\quad (\overline{1},0,2,1,0,2,\ldots)\}. \]
    To see where to put the bars, playing $R_1(2,1,0,2,1,0,\ldots)$, we have the next element is $(\overline{1},0,2,1,0,2,\ldots)$. Playing $R_1(\overline{1},0,2,1,0,2,\ldots)$, we get $(0,\overline{2},1,0,2,1,\ldots)$. Finally, playing $R_2(0,\overline{2},1,0,2,1,\ldots)$, which is the first playable part of $(0,\overline{2},1,0,2,1,\ldots)$, gives $(\overline{2},\overline{1},0,2,1,0,\ldots)$. Thus, the recurrent set of $\mathcal{O}_{(BWW)^\infty}$, with the bars, is
    \[ \{(\overline{2},\overline{1},0,2,1,0,\ldots),\quad (0,\overline{2},1,0,2,1,\ldots),\quad (\overline{1},0,2,1,0,2,\ldots)\}. \]
    One can check that these bars are consistent with the bars in the finite version in Figures \ref{fig:forest_1}, \ref{fig:forest_2}, and \ref{fig:forest_3}.

\subsection{Quasi-infinite forests $\mathcal{F}_P$}\label{subsec:quasi_forest}

    Now we introduce the \textit{quasi-infinite forest} for certain directed graphs {\it (digraphs)}, such as the opposites of the functional digraphs for Bulgarian solitaire  orbits. Recall that a {\it functional digraph} for a function $f:V \rightarrow V$ on a set
    $V$ has arcs $v \rightarrow f(v)$
    for each $v$ in $V$. Functional digraphs are the same as digraphs
    in which every vertex $v$ has outdegree one; this allows self-loops and directed $2$-cycles, but parallel arcs would violate the outdegree one condition.
    
    \begin{definition}
        Let $D = (V,A)$ be the {\it opposite} digraph of a  functional digraph, that is a digraph in which every vertex has in-degree one. Let $C$ be the largest subset of $V$ such that the induced digraph $D|_C$ is a permutation. Let us call $C$ the \textit{recurrent set} of $D$. 
    
    Define the quasi-infinite forest $\mathcal{F}_D$ to be the digraph in which the vertices are directed paths $p=(v_0\rightarrow v_1\rightarrow \ldots\rightarrow v_i)$ in $D$ such that $v_0 \in C$, and there is an arc  $p \rightarrow p'$ in $\mathcal{F}_D$ whenever $p, p'$ are related as follows: 
    \begin{equation}
        \label{typical-forest-arc}
    \begin{aligned}        
    p &= (v_0\rightarrow v_1\rightarrow \ldots\rightarrow v_i),\\
    p' &=(v_0\rightarrow v_1\rightarrow \ldots\rightarrow v_i\rightarrow v_{i+1}).
    \end{aligned}
    \end{equation}
     \end{definition}
     Figure \ref{fig:forest_ex} shows an example of a digraph $D$ and the corresponding forest $\mathcal{F}_D$. It is easy to see that $\mathcal{F}_D$ consists of $|C|$ trees rooted at the vertices in $C$. 
     
      \begin{figure}[h!]
        \centering
        \includegraphics{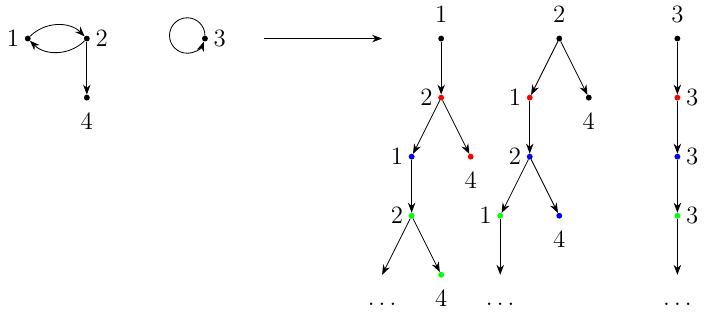}
        \caption{Digraph $D$ (left) and the corresponding quasi-infinite forest $\mathcal{F}_D$ (right)}
        \label{fig:forest_ex}
    \end{figure}

     We wish to relate two generating functions, one for the digraph $D$ and one for its quasi-infinite forest $\mathcal{F}_D$. The {\it level generating function} for $D$ is defined by 
     $$
     h(x):=\sum_{v \in V} x^{\text{level}(v)}
     $$
     where $\text{level}(v)=\min\{k : f^k(v) \in C\}$
     for the function $f:V \rightarrow V$ whose functional digraph is opposite to $D$. Letting $\ell(p):=i$ for $p=(v_0 \rightarrow v_1 \rightarrow \cdots \rightarrow v_i)$ The {\it path-length generating function} for $\mathcal{F}_D$ is defined by
     $$
     g(x):=\sum_{p} x^{\ell(p)}.
     $$
    
    \begin{lemma}\label{lem:func-graph-gf}
        For any digraph $D$ opposite to a functional digraph, one has 
        $$
        h(x) = (1-x)g(x).
        $$
    \end{lemma}

    \begin{proof}
        It is equivalent to show that
        $$
        \begin{aligned}
        g(x) &= (1+x+x^2+x^3+\cdots)h(x)\\
        &=h(x) + xh(x) +x^2h(x) + x^3h(x) \cdots
        \end{aligned}
        $$
        One can interpret each term $x^k h(x)$ on the last line as follows. Call an arc $v \rightarrow v'$ in $D$ {\it permutational} if both $v,v'$ lie in $C$ and $f(v')=v$, and {\it non-permutational} otherwise.
        It is not hard see that every path $p=(v_0 \rightarrow v_1 \rightarrow \cdots \rightarrow v_i)$ indexing a vertex in $\mathcal{F}_D$ starts with a (possibly empty) sequence of all permutational steps $v_0 \rightarrow v_1 \rightarrow \cdots \rightarrow v_k$, followed by a (possibly empty) sequence of steps $v_k \rightarrow v_{k+1} \rightarrow \cdots \rightarrow v_i$ which are all non-permutational; the index $k$ is therefore uniquely determined. This lets one decompose $\mathcal{F}_D$ into vertex subsets 
        $$
        \mathcal{F}_D = \mathcal{F}_0 \sqcup \mathcal{F}_1 \sqcup \mathcal{F}_2 \sqcup \mathcal{F}_3\sqcup \cdots
        $$ 
        where $\mathcal{F}_k$ are the vertices whose corresponding path starts with $k$ permutational steps. In Figure~\ref{fig:forest_ex}, the sets $\mathcal{F}_k$ for $k=0,1,2,3$ are colored black, red, blue, green, respectively. One then checks that, for each $k=0,1,2,\ldots$, the map sending 
        $p=(v_0 \rightarrow v_1 \rightarrow \cdots \rightarrow v_i)$ to $v_i$
        restricts to a bijection 
        $\phi_k: \mathcal{F}_k \longrightarrow V$ satisfying
        $\ell(p)=k+ \text{level}(v_i)$. Consequently,
         $$
         \sum_{p \in \mathcal{F}_k} x^{\ell(p)}= x^k h(x).\qedhere
         $$
    \end{proof}
    
    Let us recall the example at the beginning of Section \ref{subsec:BSlimit} with the primitive necklace $P=BWW$ and its powers $P^1,P^2,P^3,\cdots$. Figures \ref{fig:forest_1}, \ref{fig:forest_2}, and \ref{fig:forest_3} show the digraphs for $\mathcal{O}_{(BWW)^1},\mathcal{O}_{(BWW)^2}, \mathcal{O}_{(BWW)^3}$. Recall that each orbit $\mathcal{O}_{(BWW)^\ell}$ has a recurrent set $\{(BWW)^\ell,(WBW)^\ell,(WWB)^\ell\}$. Hence, each orbit $\mathcal{O}_{(BWW)^\ell}$ corresponds to a quasi-infinite forest $\mathcal{F}_{(BWW)^\ell}$, which is a disjoint union of three trees $\mathcal{T}_{(BWW)^\ell}$, $\mathcal{T}_{(WBW)^\ell}$, $\mathcal{T}_{(WWB)^\ell}$ rooted at $(BWW)^\ell$, $(WBW)^\ell$, $(WWB)^\ell$, respectively. Since $\mathcal{O}_{(BWW)^\ell}$ converges to a digraph $\mathcal{O}_{(BWW)^\infty} = \lim_{\ell\rightarrow\infty}\mathcal{O}_{(BWW)^\ell}$, the corresponding quasi-infinite forests $\mathcal{F}_{(BWW)^\ell}$ also converge to a quasi-infinite forest $\mathcal{F}_{BWW} = \lim_{\ell\rightarrow\infty}\mathcal{F}_{(BWW)^{\ell}}$. The recurrent set of $\{C_1,C_2,C_3\}$ of $\mathcal{O}_{(BWW)^\infty}$ is still in bijection with $\{BWW,WBW,BWB\}$,
    and hence $\mathcal{F}_{BWW}$ is a disjoint union of three trees $\mathcal{T}_{C_1},\mathcal{T}_{C_2},\mathcal{T}_{C_3}$ rooted at $C_1,C_2,C_3$ as shown in Figure \ref{fig:forest_infty}.
    
    \begin{figure}[h!]
        \centering
        \includegraphics[scale = 0.3]{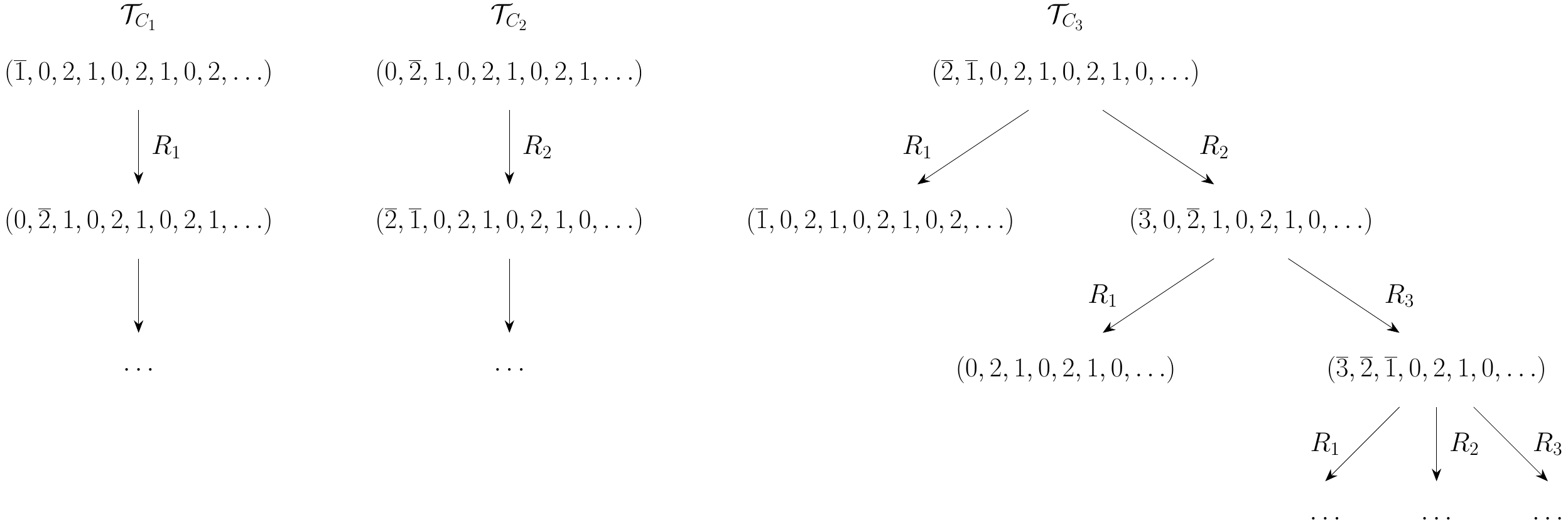}
        \caption{Part of the quasi-infinite forest $\mathcal{F}_{BWW}$}
        \label{fig:forest_infty}
    \end{figure} 
     
    Also, recall that Pham generalized this idea to every primitive necklace $P$ of length $|P|=n$. 
    She showed in \cite{pham2022limiting} that for the game digraphs $\mathcal{O}_{P^\ell}$, 
    the bijections of their recurrent sets $\{C^{(\ell)}_1,C^{(\ell)}_2,\ldots,C^{(\ell)}_n\}$
    extend to digraph inclusions
    $$
     \mathcal{O}_{P} \hookrightarrow \mathcal{O}_{P^2} \hookrightarrow \mathcal{O}_{P^3} \hookrightarrow \cdots
    $$
    which converge to a digraph $\mathcal{O}_{P^\infty}$ with recurrent set $\{C_1,C_2,\ldots,C_n\}$. Hence, the corresponding quasi-infinite forests $\mathcal{F}_{P^\ell}$ also converge to a quasi-infinite forest $\mathcal{F}_{P}$, having $n$ trees $\mathcal{T}_{C_1},\mathcal{T}_{C_2},\ldots,\mathcal{T}_{C_1}$ rooted at $C_1,C_2,\ldots,C_n$.
    In other words,  
\begin{equation}
    \label{forest-is-union-of-trees}
    \mathcal{F}_P = \bigsqcup_{i=1}^n \mathcal{T}_{C_i}. 
\end{equation}

    We can also generalize the concept of quasi-infinite trees $\mathcal{T}_{C_i}$ to trees $\mathcal{T}_\mu$ rooted at any element $\mu$ in $\BSlimit$. Observe that we can pick any element $\mu$, not necessarily a recurrent cycle element, and start playing reversed BS from $\mu$. Let $\mathcal{O}_\mu$ be the set of elements reachable from $\mu$ after a sequence of reversed BS moves. When $\mu$ is not a recurrent cycle element, then for every element $\nu$ in $\mathcal{O}_\mu$, there is a unique sequence of $i$ moves $R_{j_1},\ldots,R_{j_i}$ such that $\nu = R_{j_i}\circ\ldots\circ R_{j_1}(\mu)$. Thus, we can associate each element $\nu$ in $\mathcal{O}_\mu$ with a ``level'' $i$. Hence, we can define the tree rooted at $\mu$ and the level generating function of this tree as follows.

    \begin{definition}\label{def:subtree}
        For any element $\mu$ in $\BSlimit$, denote by $\mathcal{T}_\mu$ the {\it tree rooted at $\mu$} whose vertices are indexed by elements in $\mathcal{O}_\mu$, and there is a direct edge $\nu\rightarrow\nu'$ if $\nu' = R_j(\nu)$ for some $j$. In this case, for each $\nu$ in $\mathcal{O}_\mu$, we let $\ell(\mu,\nu)$
        denote the number of steps in the path from $\mu$ to $\nu$, and define the level generating function for $\mathcal{T}_\mu$ to be 
        $$
        g_\mu=g_\mu(x):=\sum_\nu x^{\ell(\mu,\nu)}
        $$
        where the sum runs over all such vertices $\nu$ of $\mathcal{T}_\mu$.
    \end{definition}

    For example, Figure \ref{fig:T-mu-ex} shows the tree $\mathcal{T}_\mu$ rooted at the element $\mu = (\overline{1}, \overline{2}, 1, 1, \ldots)$. The level generating function of $\mathcal{T}_\mu$ is $g_\mu(x) = 1+2x+3x^2+\ldots$.

    \begin{figure}[h!]
        \centering
        \includegraphics[scale = 0.5]{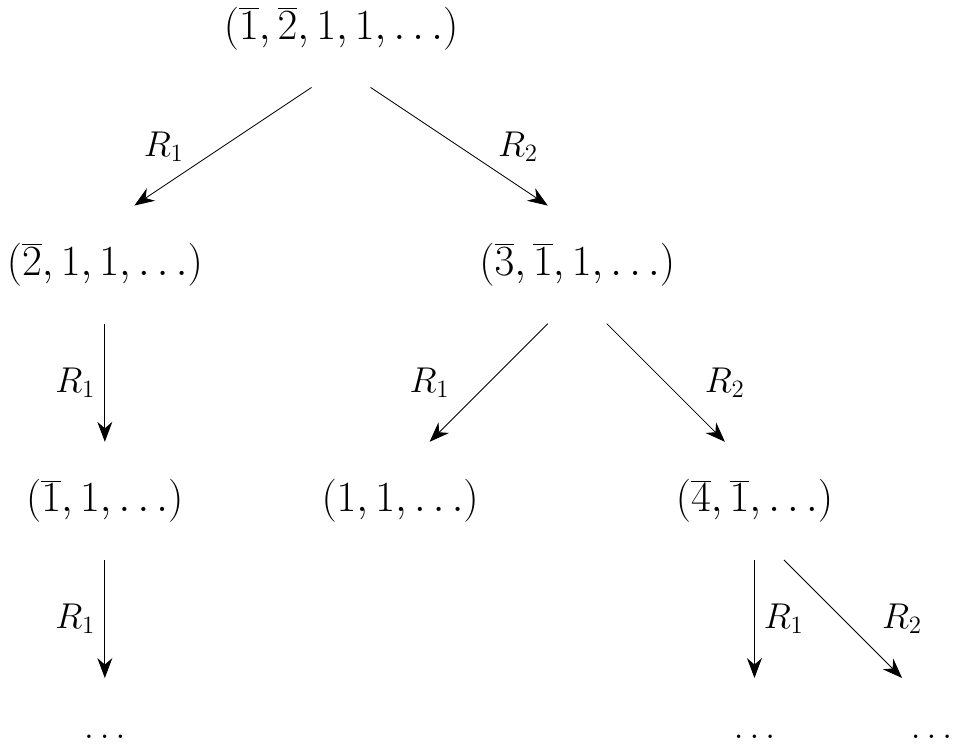}
        \caption{$\mathcal{T}_{(\overline{1}, \overline{2}, 1, 1, \ldots)}$}
        \label{fig:T-mu-ex}
    \end{figure}

    An important idea is comparing subtrees rooted at different vertices of the forest $\mathcal{F}_P$. In the special case where $\mu=C_i$ is one of the roots of the forest $\mathcal{F}_P$, so that $\mathcal{T}_\mu=\mathcal{T}_{C_i}$ is one of the trees in the forest, denote $g_\mu(x)$  by $g_i=g_i(x)=g_{C_i}(x)$. Thus, \eqref{forest-is-union-of-trees} shows that
    $$
    g(x)=\sum_{i=1}^n g_i(x)=g_1(x) + \cdots + g_n(x)
    $$
    and the work of  \cite{pham2022limiting} (or Lemma~\ref{lem:func-graph-gf} above) shows that
    \[ H_P(x) = (1-x)g(x). \]
   Thus, the key step in studying $H_P(x)$ is to understand the generating functions 
   $\{ g_i(x)\}_{i=1}^n$. Our strategy to prove Theorem~\ref{thm:BWW_WBB} is to relate the $\{ g_i(x)\}_{i=1}^n$ via a linear system of equations. Our strategy to prove Theorem~\ref{thm:BWB_WBW} uses the following notion.
   \begin{definition}
   \label{subtree-isomorphism}
       Say that two quasi-infinite trees $\mathcal{T}_\mu$ and $\mathcal{T}_\nu$ are \textit{isomorphic} if there is a bijection $f$ between their vertices that respects the reverse BS moves, i.e., one has $\rho' = R_i(\rho)$ in $\mathcal{T}_\mu$ if and only if $f(\rho') = R_i(f(\rho))$ in $\mathcal{T}_\nu$.
    \end{definition}

%%%%%%%%%%%%%%%%%%%%%%%%%%%%%%%%%%%%%%%%%%%%%%%%%%%%%%%%%%%%%%%%%%%%
\section{Fuses and pre-fuses}\label{sec:fuse}

\subsection{$k$-fuses}\label{subsec:k-fuse}

    In this section, we introduce the concept of \textit{$k$-fuse} that shows up in almost every quasi-infinite forest. In general, we say an element $\mu=(\mu_1,\ldots,\mu_k,\mu_{k+1},\ldots)$ in $\BSlimit$ contains a $k$-fuse if its first $k$ parts $\mu_1,\ldots,\mu_k$ satisfy the conditions in Definition \ref{def:k-fuse} below. We then view $\mu$ as $(\mu_1,\ldots,\mu_k,\nu)$ where $\nu$ is some other element of $\BSlimit$, i.e. $\mu$ is $\nu$ following a prefix $(\mu_1,\ldots,\mu_k)$. We will eventually show that regardless of the exact values of $\mu_1,\ldots,\mu_k$, as long as they satisfy the conditions in Definition \ref{def:k-fuse}, we have 
    \begin{equation}
    \label{desired-equation}
    g_{\mu}(x) = u_k(x) \cdot g_{\nu}(x),
    \end{equation}
    where $u_k(x)$ only depends on $k$ and does not depend on $\mu$ or $\nu$ or the exact values of $\mu_1,\ldots,\mu_k$. Then we will combinatorially interpret the coefficients of $u_k(x)$.

    \begin{definition}\label{def:k-fuse}
        We say $(\mu_1,\ldots,\mu_k)$ is a \textbf{$k$-fuse} if 

        \begin{enumerate}
            \item $\mu_1,\mu_2,\ldots,\mu_{k-1}$ are either $1$ or $2$, but $\mu_k \geq 3$,
            \item all parts $\mu_1,\mu_2,\ldots,\mu_k$ are playable, and
            \item for $1 \leq j\leq k-1$, if $\mu_j=1$ then  $\mu_{j+1} \neq 1$, i.e. there is no two consecutive ones.
        \end{enumerate}

        If $\mu = (\mu_1,\ldots,\mu_k,\ldots)$, that is, the first $k$ parts of $\mu$ are $\mu_1,\ldots,\mu_k$, and $(\mu_1,\ldots,\mu_k)$ is a $k$-fuse, we say that $\mu$ contains a $k$-fuse.
    \end{definition}

    \begin{example}
        Both $\mu=(\bar{2},\bar{1},\bar{3},\bar{1},\bar{2},\ldots)$ and $\nu=(\bar{1},\bar{2},\bar{3},\bar{2},\bar{2},\ldots)$
        contain $3$-fuses.
    \end{example}

    \begin{remark}\label{rem:k-fuse}
        We explain here why we call $(\mu_1,\ldots,\mu_k)$ a ``$k$-fuse".  First of all, once we play any of the first $k$ parts, all parts after $\mu_k$, i.e. $\mu_{k+1},\mu_{k+2},\ldots$, are no longer playable because $\mu_k\geq 3$. This is because Lemma \ref{lem:bs_move_lim} says that if we let $\mu' = R_j(\mu)$, then for $i\geq j$, $\mu'_i$ is playable only if $\sum_{s = j}^i \mu_i < 3$. Having $\mu_k\geq 3$ immediately violates this condition. As a result, once we play any of the first $k$ parts, only the first $k$ parts are playable. In addition, Proposition \ref{prop:k-fuse-len} will show that the reversed BS game will terminate after at most $k$ moves. For example, figure \ref{fig:sep-iso-ex} shows two different elements of $\BSlimit$ that both contain a $3$-fuse. Although the $3$-fuses are different, the branches after playing $R_1$, $R_2$, or $R_3$ are isomorphic. Furthermore, for every element in these branches, the playable parts are only those initially in the $3$-fuses, and after at most $3$ steps, there is no more playable parts. We can think of a $k$-fuse as a fuse of a ``bomb''. Once we ``trigger'' the fuse by playing one of the first $k$ parts, there is nothing we can do except continue ``burning'' the fuse. Eventually, after at most $k$ moves, the bomb ``explodes'' and there is no more possible move.

        \begin{figure}[h!]
    		
         \centering
            \begin{subfigure}[b]{0.45\textwidth}\label{subfig:3-sep-1}
                \centering
        
                \includegraphics[width = \textwidth]{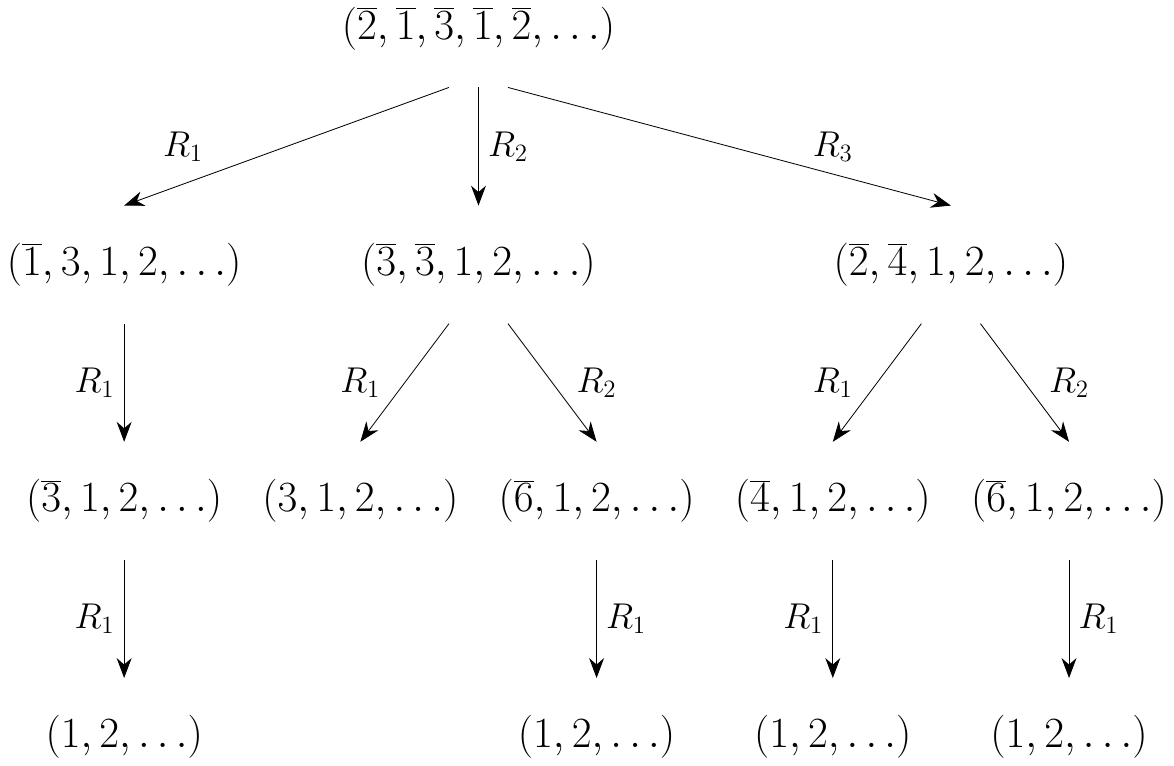}
            \end{subfigure}
         \quad\quad
            \begin{subfigure}[b]{0.45\textwidth}\label{subfig:3-sep-2}
                \centering
        
                \includegraphics[width = \textwidth]{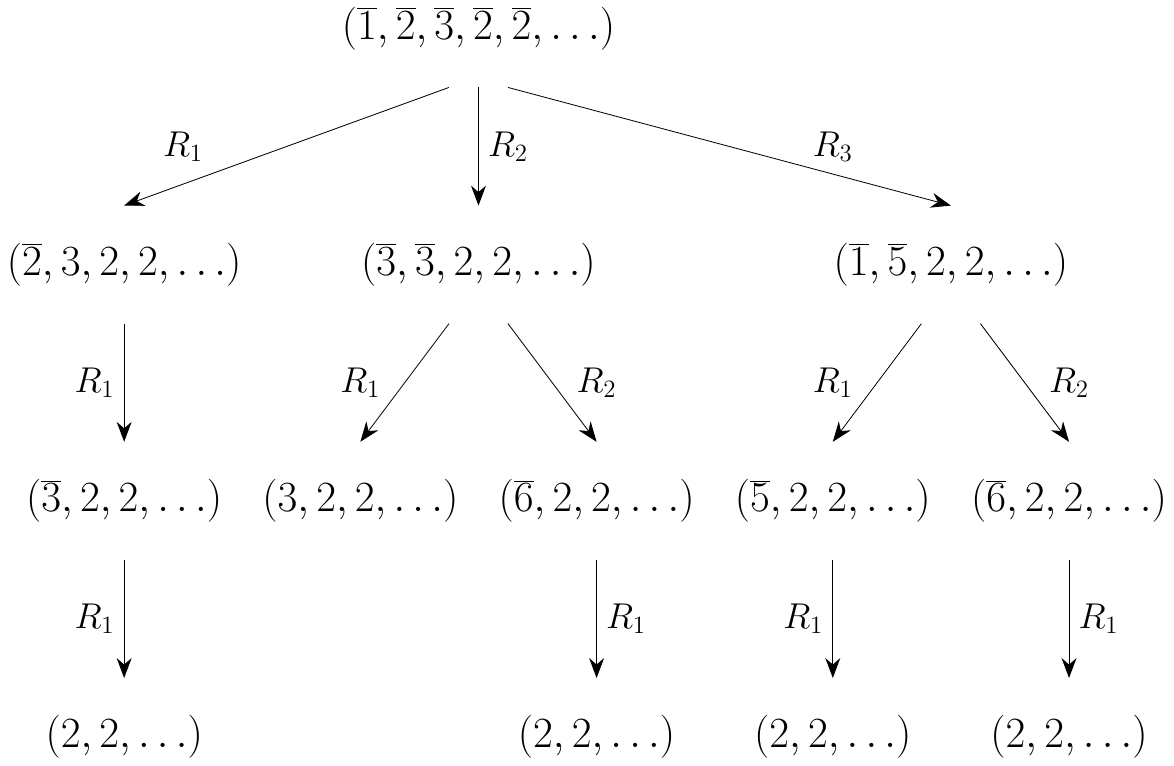}
            \end{subfigure}
            
            \caption{3-fuses}
            \label{fig:sep-iso-ex}
            
        \end{figure}
    \end{remark}

    Let us now make some of the earlier comments about
    \eqref{desired-equation} more precise. Given a subset $A\subseteq \{1,2,3,\ldots\}$, define
    \[ R_A(\mu):=\{\rho=(R_{j_1}\circ R_{j_2} \circ \cdots R_{j_k})(\mu)\text{ for some }j_1,j_2,\ldots,j_k \in A\}. \]
    Given $\mu=(\mu_1,\ldots,\mu_k,\nu)$ starting with a $k$-fuse, let $V:=R_{ \{k+1,k+2,\ldots\} }(\mu)$. Corollary \ref{cor:all-has-fuse} below will show that every element in $V$ also has a $k$-fuse. By Remark \ref{rem:k-fuse} and Proposition \ref{prop:k-fuse-len}, the tree $\mathcal{T}_\mu$ has a disjoint decomposition
    \[ \mathcal{T}_\mu = \bigsqcup_{\rho \in V} R_{ \{ 1,2,\ldots,k \} }(\rho). \]
    Figure \ref{subfig:u_2_sep} shows the disjoint decomposition of the tree rooted at $\mu = (\overline{1},\overline{3},\overline{1},\overline{2},1,1,\ldots)$ with a $2$-fuse. Each component $R_{ \{ 1,2,\ldots,k \} }(\rho)$ is illustrated via the color-coding.

    \begin{figure}[h!]
        
     \centering
        \begin{subfigure}[b]{0.4\textwidth}
            \centering
    
            \includegraphics[width = \textwidth]{fig/u_2_core.pdf}
            \caption{}
            \label{subfig:u_2_core}
        \end{subfigure}
     \quad\quad
        \begin{subfigure}[b]{0.5\textwidth}
            \centering
    
            \includegraphics[width = \textwidth]{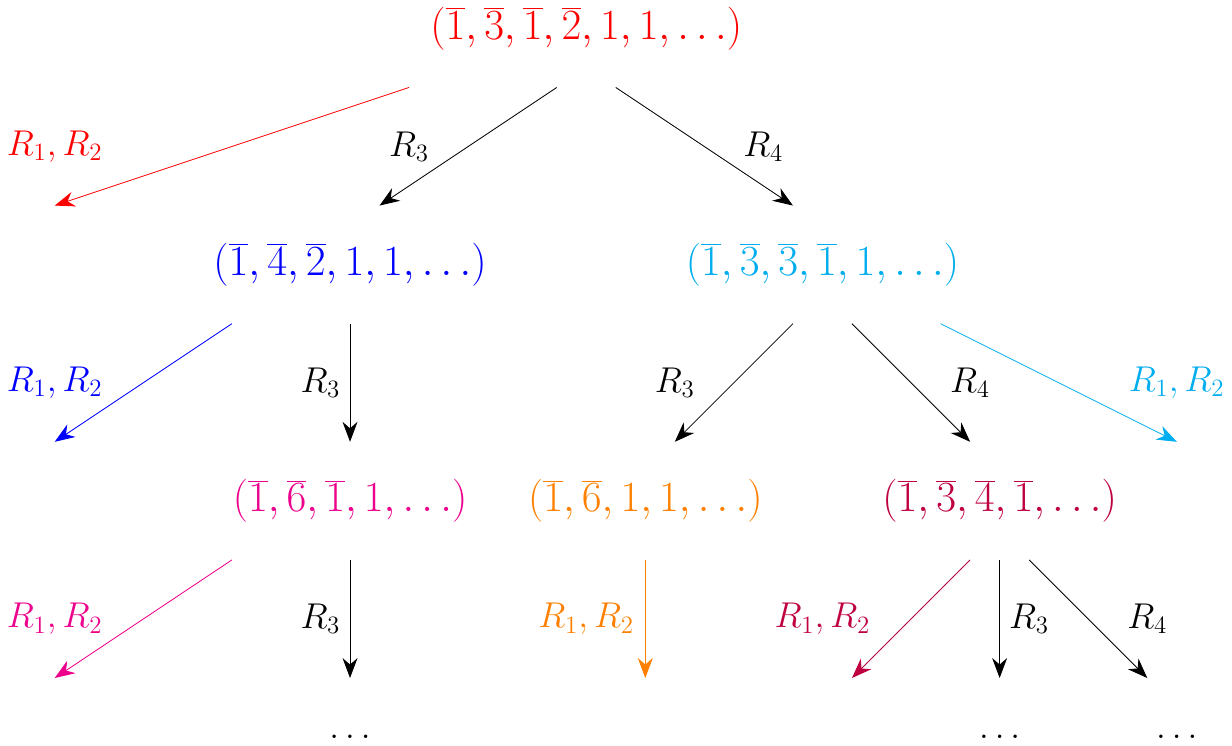}
            \caption{}
            \label{subfig:u_2_sep}
        \end{subfigure}
        
        \caption{}
        \label{fig:u_2-ex}
        
    \end{figure}

    Because of the disjoint decomposition, one can write
    \[ g_\mu(x) = \sum_{\rho \in V} x^{\ell(\mu,\rho)} u_\rho(x) \]
    where
    \[ u_\rho(x) = \sum_{\sigma \in R_{ \{ 1,2,\ldots,k \} (\rho)}} x^{\ell(\rho,\sigma)}. \]
    Proposition~\ref{prop:iso-k-branches} below shows that all of the subtrees $R_{ \{ 1,2,\ldots,k \} }(\rho)$ are isomorphic.
    Hence one can define a single (polynomial) generating function $u_k(x):=u_\rho(x)$ for all $\rho \in V$, to be studied further in Section~\ref{subsec:formula-u-k} below.
    Consequently,
 \[
 g_\mu(x) = u_k(x) \cdot \sum_{\rho \in V} x^{\ell(\mu,\rho)}.
 \]
 Our next proposition shows $\mathcal{T}_\mu |_V$ is isomorphic to $\mathcal{T}_\nu$, which will eventually imply \eqref{desired-equation}.
 
    \begin{prop}\label{prop:k-sep-isomorphic}
        Let $\mu = (\overline{\mu_1},\ldots,\overline{\mu_k},\nu)$ where $(\overline{\mu_1},\ldots,\overline{\mu_k})$ is a $k$-fuse. Then the vertex set $V := R_{\{k+1,k+2,\ldots\}]}(\mu)$ has $\mathcal{T}_\mu |_V$ isomorphic to $\mathcal{T}_\nu$.
    \end{prop}

    \begin{proof}
        We will show that a sequence of moves $R_{i_1},R_{i_2},\ldots,R_{i_j}$ is possible from $\nu$ if and only if the sequence of moves $R_{k+i_1},R_{k+i_2},\ldots,R_{k+i_j}$ is possible from $\mu$. In fact, we will prove a slightly stronger statement: $\rho := R_{i_j}\circ\ldots\circ R_{i_1}(\nu)$ exists if and only if $R_{k+i_j}\circ\ldots\circ R_{k+i_1}(\mu)$ also exists, and $R_{k+i_j}\circ\ldots\circ R_{k+i_1}(\mu) = (\overline{\mu_1},\ldots,\overline{\mu_{k-1}},\overline{\mu_k'},\rho)$ where $(\overline{\mu_1},\ldots,\overline{\mu_{k-1}},\overline{\mu_k'})$ is a $k$-fuse. We will prove this by induction on $j$. The base case where $j = 0$ is obvious. 
        
        Suppose $\rho := R_{i_j}\circ\ldots\circ R_{i_1}(\nu)$ exists and $\sigma := R_{k+i_j}\circ\ldots\circ R_{k+i_1}(\mu) = (\overline{\mu_1},\ldots,\overline{\mu_{k-1}},\overline{\mu_k'},\rho)$. Suppose $\rho' = R_{i_{j+1}}(\rho)$ exists for some $i_{j+1}>1$, then since $i_{j+1}>1$, this move only affects and depends on parts $i_{j+1}-1,i_{j+1},i_{j+1}+1 \ldots$ in $\rho$. These parts are identical to parts $k+i_{j+1}-1,k+i_{j+1},k+i_{j+1}+1 \ldots$ in $\sigma$. Thus, $\sigma' = R_{k+i_{j+1}}(\sigma)$ exists, and $\sigma' = (\overline{\mu_1},\ldots,\overline{\mu_{k-1}},\overline{\mu_k'},\rho')$.

        Finally, suppose $\rho' = R_{1}(\rho)$ exists. By Lemma \ref{lem:bs_move_result}, $\rho'$ is obtained by removing the first part of $\rho$ and putting the bars on the remaining parts following the rules in Lemma \ref{lem:bs_move_result}. On the other hand, $\sigma' = R_{k+1}(\sigma)$ is obtained from $\sigma$ by adding $\rho_1$ to $\mu_k'$ and putting the bars on the remaining parts also following the rules in Lemma \ref{lem:bs_move_result}. Again, parts $1,2,\ldots$ of $\rho$ are the same as parts $k+1,k+2,\ldots$ of $\sigma$, so the extra bars are put on respective parts. Thus, $\sigma' = (\overline{\mu_1},\ldots,\overline{\mu_{k-1}},\overline{\mu_k''},\rho')$ where $\mu_k'' = \mu_k' + \rho_1$. Note that this does not violate the conditions of $k$-fuses, i.e. $(\overline{\mu_1},\ldots,\overline{\mu_{k-1}},\overline{\mu_k''})$ is still a $k$-fuse.

        The argument for the converse is exactly the same.
    \end{proof}

    The following corollary is immediate from the proof of Proposition \ref{prop:k-sep-isomorphic}.

    \begin{cor}\label{cor:all-has-fuse}
        There is an isomorphism from $\mathcal{T}_\nu$ to $\mathcal{T}_\mu |_V$ that maps every element $\rho\in \mathcal{T}_\nu$ to an element $(\overline{\mu_1},\ldots,\overline{\mu_{k-1}},\overline{\mu_k'},\rho)\in \mathcal{T}_\mu |_V$ where $(\overline{\mu_1},\ldots,\overline{\mu_{k-1}},\overline{\mu_k'})$ is a $k$-fuse. In particular, every element in $\mathcal{T}_\mu |_V$ has a $k$-fuse.
    \end{cor}

    Now we prove that $u_\rho(x)$ are the same for all $\rho\in V$.
    
    \begin{prop}\label{prop:iso-k-branches}
        Let $\rho = (\overline{\rho_1},\ldots,\overline{\rho_k},\nu)$ where $(\rho_1,\ldots,\rho_k)$ is a $k$-fuse. Then for all values of $\rho_1,\ldots,\rho_k$ (satisfying the conditions of $k$-fuses) and for all $\nu$, the subtrees $R_{\{1,2,\ldots,k\}}(\rho)$ are isomorphic.
    \end{prop}

    \begin{proof}
        We will prove this by induction on $k$. The base case where $k = 1$ is obvious. Suppose the statement is true for $k = 1,\ldots,j-1$, consider any two elements $\rho = (\overline{\rho_1},\ldots,\overline{\rho_j},\nu)$ and $\rho' = (\overline{\rho'_1},\ldots,\overline{\rho'_j},\nu')$ where $(\overline{\rho_1},\ldots,\overline{\rho_j})$ and $(\overline{\rho'_1},\ldots,\overline{\rho'_j})$ are $j$-fuses. Let $\sigma = R_i(\rho)$ and $\sigma' = R_i(\rho')$ for some $i\leq j$, we will prove that $\mathcal{T}_\sigma$ and $\mathcal{T}_{\sigma'}$ are isomorphic. 
        
        First, let $V = R_{\{i,\ldots\}}(\sigma)$ and $V' = R_{\{i,\ldots\}}(\sigma')$, we claim that $\mathcal{T}_\sigma |_V$ and $\mathcal{T}_\sigma' |_{V'}$ are isomorphic. In fact, we claim that $V = \{\pi~|~\pi = R_{i}^m(\sigma), 0\leq m\leq k-i\}$. This is because in $\sigma$, $\sigma_i$ is playable since $\rho_i < 3$, but $\sigma_r$ is not playable for all $r >i$ since $\sum_{s = i}^r\rho_s \geq \rho_i + \rho_{i+1} \geq 3$ by condition 3 in Definition \ref{def:k-fuse}. For the same reason, in $R_{i}^m(\sigma)$ for $0\leq m < k-i$, the $i$th part is playable but any part after that is not. However, in $R_{i}^{k-i-1}(\sigma)$, the $i$th part is $\rho_j$, which is at least $3$. Thus, in the $R_{i}^{k-i}(\sigma$, the $i$th part is also not playable. Thus, $V = \{\pi~|~\pi = R_{i}^m(\sigma), 0\leq m\leq k-i\}$. Similarly, $V' = \{\pi~|~\pi = R_{i}^m(\sigma'), 0\leq m\leq k-i\}$. Hence, $\mathcal{T}_\sigma |_V$ and $\mathcal{T}_\sigma' |_{V'}$ are isomorphic.

        Finally, every element $\pi$ in $\mathcal{T}_\sigma |_V$ and $\mathcal{T}_\sigma' |_{V'}$ contains an $(i-1)$-fuse. Since $i-1\leq j-1$, by the inductive hypothesis, $R_{\{1,\ldots,i-1\}}(\pi)$ are isomorphic for all $\pi$ in $\mathcal{T}_\sigma |_V$ and $\mathcal{T}_\sigma' |_{V'}$. This completes the proof.
        
    \end{proof}

    The proof of Proposition \ref{prop:iso-k-branches} also suggests the following result.

    \begin{prop}\label{prop:k-fuse-len}
         Let $\rho = (\overline{\rho_1},\ldots,\overline{\rho_k},\nu)$ where $(\rho_1,\ldots,\rho_k)$ is a $k$-fuse. Let $\alpha_i$ be $R_i(\rho)$ for $1\leq i \leq k$. Let $\mathcal{T}_{[k],\rho} := \bigcup_{1\leq i\leq k} \mathcal{T}_{\alpha_i}$. Let $V$ be the set of elements in $\mathcal{T}_{[k],\rho}$, then
         \[ V = \{ \pi~|~\pi = (R_{i_j} \circ \ldots \circ R_{i_1})(\rho) \} \]
         where $j\leq k$ and $k\geq i_1 \geq i_2 \geq \ldots \geq i_j$. Specifically, $V = R_{\{1,\ldots,k\}}(\rho)$, and $\mathcal{T}_\rho|_V$ has depth $k$.
    \end{prop}

    \begin{proof}
        We will prove this by induction on $k$. If $k = 1$, then $\rho = (\overline{\rho_1},\nu)$ where $\rho_1 \geq 3$. Since $\rho_1 \geq 3$, in $R_1{\rho}$, no part is playable. Thus, $\mathcal{T}_{[1],\rho} = \mathcal{T}_{\alpha_1}$ only contains one element: $R_1(\rho)$, so the statement is true for $k = 1$.

        If $k>1$, consider any $\mathcal{T}_{\alpha_i}$ with $1 \leq i \leq k$. The proof of Proposition \ref{prop:iso-k-branches} shows that
        \[ R_{\{i,\ldots\}}(\alpha_i) = \{\pi~|~\pi = R_i^m(\alpha_i), 0 \leq m \leq k-i\} = \{\pi~|~\pi = R_i^m(\rho), 1 \leq m \leq k-i+1\}. \]
        Furthermore, every element $\sigma$ in $R_{\{i,\ldots\}}(\alpha_i)$ has an $(i-1)$-fuse, by induction, the elements in $\mathcal{T}_{[i-1],\sigma}$ have the form
        \[ (R_{i_j} \circ \ldots \circ R_{i_1})(\sigma) \]
        where $j\leq i-1$ and $i-1\geq i_1\geq i_2\geq \ldots\geq i_j$. Hence, every element in $\mathcal{T}_{\alpha_i}$ has the form
        \[ (R_{i_j} \circ \ldots \circ R_{i_1} \circ R_{i}^m) \]
        where $j\leq i-1$, $i-1\geq i_1\geq i_2\geq \ldots\geq i_j$, and $0 \leq m \leq k-i+1$. Thus, the statement is true.

        This proves that $V \subseteq R_{\{1,\ldots,k\}}(\rho)$. Clearly, we also have $R_{\{1,\ldots,k\}}(\rho) \subseteq V$, so $V = R_{\{1,\ldots,k\}}(\rho)$. Finally, to show that $\mathcal{T}_\rho|_V$ has depth $k$, it suffices to check that $R_1^k(\rho)$ exists, which is not difficult.
    \end{proof}
    
    For example, in Figure \ref{fig:sep-iso-ex}, we have two elements $\rho = (\overline{2},\overline{1},\overline{3},\overline{1},\overline{2},\ldots)$ and $\rho' = (\overline{1},\overline{2},\overline{3},\overline{2},\overline{2},\ldots)$, both containing 3-fuses. Even though the exact values of the two 3-fuses are different, and the remaining parts are also different ($(\overline{1},\overline{2},\ldots)$ and $(\overline{2},\overline{2},\ldots)$), still $R_{\{1,2,3\}}(\rho)$ and $R_{\{1,2,3\}}(\rho')$ are isomorphic and both have depth $3$.
    
    Proposition \ref{prop:iso-k-branches} means that $R_{\{1,2,\ldots,k\}}(\rho)$ only depends on $k$, and hence the level generating function of this subtree, denote $u_k(x)$, also depends only on $k$. Proposition \ref{prop:k-fuse-len} shows that $u_k(x)$ has degree $k$. For instance, Figure \ref{fig:sep-iso-ex} shows that $u_3(x) = 1+3x+5x^2+4x^3$. Furthermore, if $\rho$ is an element at level $i$ of some tree $\mathcal{T}_\mu$, then the elements in this subtree contribute exactly $u_k(x)x^i$ to the level generating function $g_\mu$. Thus, we say that $\rho$ has a \textit{coefficient} $u_k(x)$. Combining Propositions \ref{prop:k-sep-isomorphic} and \ref{prop:iso-k-branches} we achieve the desired equation \eqref{desired-equation}.

    \begin{cor}\label{cor:k-sep-genfunc}
       If $\mu = (\overline{\mu_1},\ldots,\overline{\mu_k},\nu)$ where $(\overline{\mu_1},\ldots,\overline{\mu_k})$ is a $k$-fuse then the generating functions $g_{\mu}(x)$ and $g_{\nu}(x)$ of $\mathcal{T}_{\mu}$ and $\mathcal{T}_{\nu}$ are related by
        \[ g_{\mu}(x) = u_k(x)\cdot g_{\nu}(x) \]
        where $u_k(x)$ only depends on $k$. 
    \end{cor}

    Figure \ref{fig:u_2-ex} shows an example of Corollary \ref{cor:k-sep-genfunc}. Figure \ref{subfig:u_2_core} shows the tree of an element $\nu = (\overline{1}, \overline{2},1,1,\ldots)$, and figure \ref{subfig:u_2_sep} shows the tree of an element $\mu$ that consists of a $2$-fuse followed by $\nu$. In $\mathcal{T}_\mu$, if $R_1$ and $R_2$ are not played, the elements are exactly the elements in $\mathcal{T}_\nu$. However, at each element $\rho$, one can play $R_1$ or $R_2$ and get to $R_{\{1,2\}}(\rho)$. Thus, each element has a coefficient $u_2(x)$.

    \subsection{Combinatorial formula for $u_k(x)$}\label{subsec:formula-u-k}
        
    In later sections, we will see that these coefficients $u_k(x)$ are very crucial, especially for computing the generating function $H_P(x)$. Fortunately, these coefficients can be described combinatorially through weak compositions. Recall that a {\it weak composition}
    $\alpha=(\alpha_1,\alpha_2,\ldots,\alpha_r)$ of $k$ is a sequence of nonnegative integers $\alpha_i$ with $\alpha_1+\cdots+\alpha_r=k$.

    \begin{prop}\label{prop:u_n_formula}
        For all $k$,  
        \[ u_k(x) = \sum_{i=0}^k c_{i,k-i}x^i \]
        where $c_{n,i}$ is the number of weak compositions of $n$ with exactly $i$ zeros.
    \end{prop}

    \begin{proof}
        Let $\mu = (\overline{\mu_1},\ldots,\overline{\mu_k})$ be an arbitrary $k$-fuse, we will construct a bijection between weak compositions of $i$ with $k-i$ zeros and elements at level $i$ in $\mathcal{T}_\mu$ recursively.  Given a weak composition $(\nu_1,\ldots\nu_\ell)$ of $i$ with $k-i$ zeros, we obtain the corresponding element as follows:

        \begin{enumerate}
            \item If $\nu_1 = \ldots = \nu_\ell = 0$, do nothing and stop. Note that this corresponds to $\mu$, the only element at level $0$, and also corresponds to the only composition of $0$ with $k$ zeros.
            \item Else, there is a largest index $\ell - m$ such that $\nu_{\ell - m}\neq 0$. Then let $p = k-m-\nu_{\ell-m}+1$, and play $R_{p}$ repeatedly $\nu_{\ell - m}$ times. Note that after this, we have a $(p-1)$-fuse and the weak composition $(\nu_1,\ldots,\nu_{\ell-m-1})$; repeat the process.
        \end{enumerate}

        First, observe that after step (2), the remaining weak composition $(\nu_1,\ldots,\mu_{\ell-m-1})$ is a weak composition of $i-\nu_{\ell - m}$ with $n-i - m$ zeros, and the remaining fuse is a $(k-m-\nu_{\ell - m})$-fuse. Since $(i-\nu_{\ell - m}) + (k-i - m) = k-m-\nu_{\ell - m}$, the recursion is well-defined.

        It is easy to see that two different weak compositions define two different playing sequences, and injectivity follows since no two playing sequences lead to the same element by nature of Bulgarian Solitaire.

        Finally, in order to prove surjectivity, we define the inverse function. For each element $\mu$ in the tree, there is a unique sequence of play $(i_1,i_2,\ldots,i_t)$ that yields $\mu$ from the $k$-fuse. Moreover, by the argument in the proof of Proposition \ref{prop:k-sep-isomorphic}, this sequence is weakly decreasing. Thus, we can rewrite the play sequence as $(i_1^{\alpha_1},\ldots. i_s^{\alpha_s})$ where $(i_1,\ldots,i_s)$ is strictly decreasing and $\alpha_j\leq i_{j-1} - i_j$. Now we fill in the parts of the weak composition from right to left. For each $i_j^{\alpha_j}$, we fill in $(\alpha_j,0,\ldots,0)$ with $i_{j-1} - i_j + 1$ zeros (here we take $i_0 = k$). Finally, we fill the rest with zeros, if necessary. It is easy to check that this is the inverse of step (1) and (2) above.
    \end{proof}

    Figure \ref{fig:sep-bij-ex} shows an example of this bijection. Take the composition $(2,1)$ for instance, the $1$ means that we start by playing $R_3$ once. Then we are left with the $2$-fuse $(\overline{2},\overline{4})$ and the composition $2$. This tells us that we play $R_1$ twice, and hence we obtain the empty element after the sequence $R_3,R_1,R_1$.

    \begin{figure}[h!]
		
     \centering
        \begin{subfigure}[b]{0.4\textwidth}
            \centering
            \includegraphics[width = \textwidth]{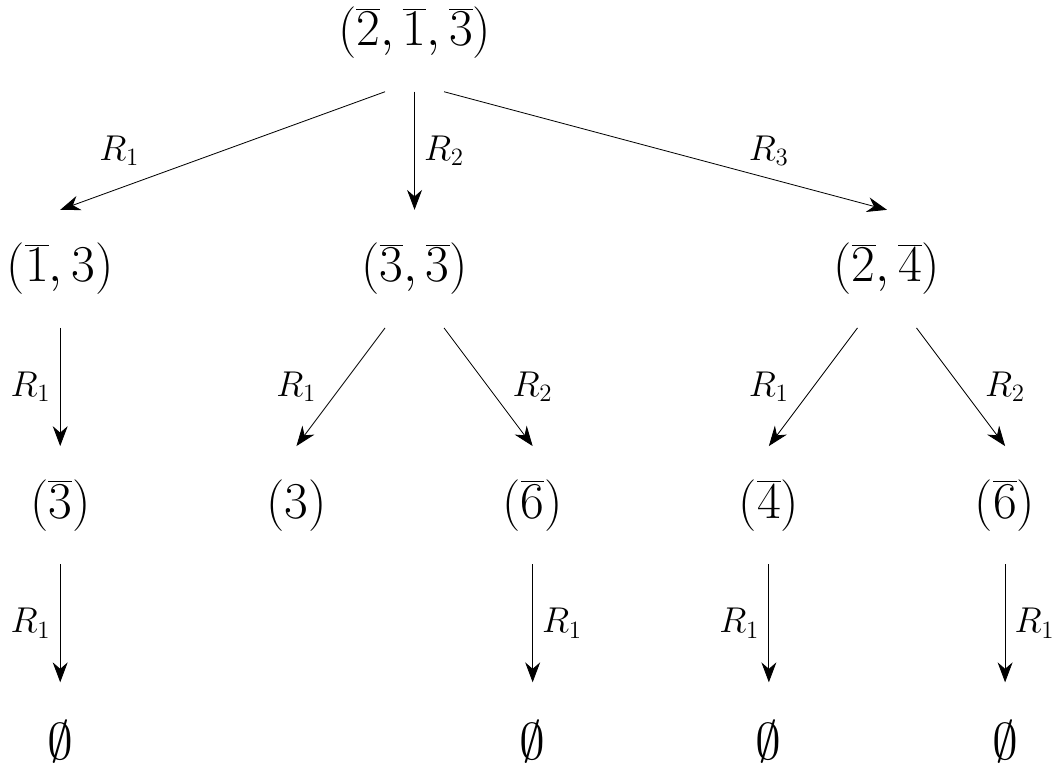}
            \caption{Tree of a 3-fuse}
            \label{subfig:3-sep-ex}
        \end{subfigure}
     \quad\quad
        \begin{subfigure}[b]{0.4\textwidth}
            \centering
            \includegraphics[width = \textwidth]{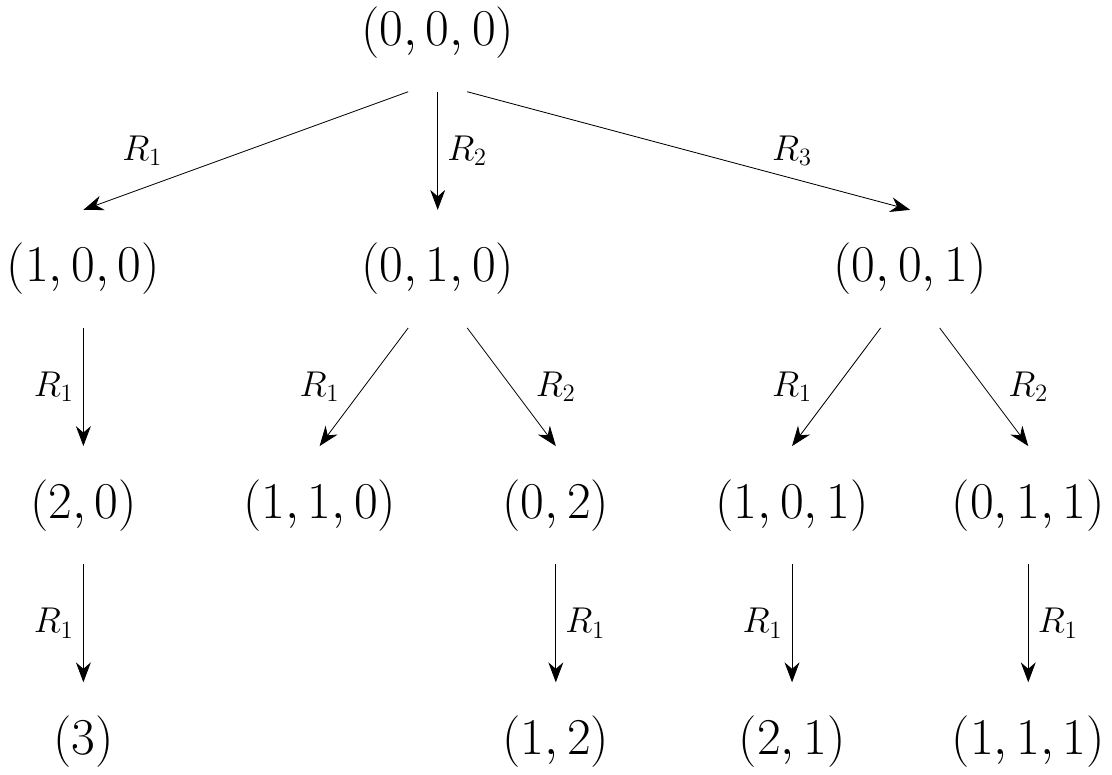}
            \caption{Corresponding compositions}
            \label{subfig:3-sep-com}
        \end{subfigure}
        
        \caption{Tree of a 3-fuse and the corresponding compositions}
        \label{fig:sep-bij-ex}
        
    \end{figure}

    \begin{remark}\label{rem:composition}
        Although there is no simple explicit formula for $c_{k,i}$ that we know of, there is a nice family of generating functions for these numbers. Fixing $i$, one has
        \[ \sum_{k=0}^\infty c_{k,i}x^k = \left(\dfrac{1-x}{1-2x}\right)^{i+1}. \]
        Observe that when $i = 0$,
        \[ \sum_{k=0}^\infty c_{k,0}x^k = \dfrac{1-x}{1-2x} = 1 + \sum_{k=1}^{\infty}2^{k-1}x^k, \]
        which is indeed the generating function for the number of strong compositions.
    \end{remark}

\subsection{$k$-pre-fuses}\label{subsec:k-pre-fuse}

    Now we briefly discuss $k$-pre-fuses, which will be discussed in more detail in Section \ref{sec:thm-1.2}.

    \begin{definition}\label{def:k-pre-fuse}
        We say $(\mu_1,\ldots,\mu_k)$ is a \textbf{$k$-pre-fuse} if they satisfy

        \begin{enumerate}
            \item $\mu_1,\mu_2,\ldots,\mu_{k}$ are either $1$ or $2$,
            \item all parts $1,2,\ldots,k$ are playable, and
            \item for all $j\leq k-1$, $\mu_j=1$ implies $\mu_{j+1} \neq 1$, i.e. there is no consecutive ones.
        \end{enumerate}

        If $\mu = (\mu_1,\ldots,\mu_k,\ldots)$, that is, the first $k$ parts of $\mu$ are $\mu_1,\ldots,\mu_k$, and $(\mu_1,\ldots,\mu_k)$ is a $k$-pre-fuse, we say that $\mu$ contains a $k$-pre-fuse.
    \end{definition}

    The only difference between this definition and Definition \ref{def:k-fuse} is that $\mu_k$ is also less than 3, so this is not a $k$-fuse. However, if we play any $R_i$ for $2\leq i\leq k$, we immediately reach an $(i-1)$-fuse.

\section{$B(WB)^k$ and $W(BW)^k$}\label{sec:thm-1.1}

    Now, that we have a good understanding of the $k$-fuses, we are set to prove Theorem \ref{thm:BWB_WBW},
    asserting
        $H_{B(WB)^k}(x) = H_{W(BW)^k}(x)$ for $k\geq 1$.   
    
    Recall from Section \ref{subsec:quasi_forest} that it suffices to study the generating functions $g_i$'s corresponding to the trees $\mathcal{T}_{C_i}$ where $C_i$'s are the recurrent cycle elements. Corollary \ref{cor:k-sep-genfunc} tells us that if in $\mathcal{T}_{C_i}$ there is an element $\mu$ with a $k$-fuse followed by $C_j$ for some $j$ (not necessarily different from $i$), then the weight of the subtree rooted at this element is $u_k(x)\cdot g_j(x)$. Thus, we can degenerate the whole subtree to one element representing the subtree with weight $u_k(x)\cdot g_j(x)$. We call the tree obtained from $\mathcal{T}_{C_i}$ by degenerating all such subtrees to single elements the \textbf{degenerate} tree of $\mathcal{T}_{C_i}$. For two necklaces $P$ and $P'$, we say two quasi-infinite trees $\mathcal{T}_{C_i}$ and $\mathcal{T}_{C'_i}$ are \textbf{almost isomorphic} if their degenerate trees are isomorphic, and if a degenerated element in $\mathcal{T}_{C_i}$ has weight $u_k(x)g_{C_j}(x)$ then the corresponding element in $\mathcal{T}_{C'_i}$ is also degenerated and has weight $u_k(x)g_{C'_j}(x)$. Note that two trees being almost isomorphic means that the subtrees that consist of the non-degenerate elements are isomorphic.

    For example, recall from Section \ref{subsec:quasi_forest} the quasi-infinite forest $\mathcal{F}_{WBW}$ with three trees rooted at the recurrent cycle elements
    \[ \{C_1,C_2,C_3\} = \{(\overline{1},0,2,1,0,2,\ldots),\quad (0,\overline{2},1,0,2,1,\ldots),\quad (\overline{2},\overline{1},0,2,1,0,\ldots)\} \]
    as shown in Figure \ref{subfig:F_BWW_full}. Observe that $R_1(C_1)$ is $C_2$, so we degenerate the whole subtree rooted at $R_1(C_1)$ to an element with weight $g_{C_2}(x)$. Similarly, we degenerate the whole subtree rooted at $R_1(C_2)$ to an element with weight $g_{C_3}(x)$, and the whole subtree rooted at $R_1(C_3)$ to an element with weight $g_{C_1}(x)$. Finally, $R_2(C_3)$ is a $1$-fuse followed by $C_1$, so we degenerate the whole subtree rooted at $R_2(C_3)$ to an element with weight $u_1(x)g_{C_2}(x)$. The degenerated forest is shown in Figure \ref{subfig:F_BWW_degen}.

    \begin{figure}[h!]
        \centering
        \begin{subfigure}[b]{\textwidth}
            \centering
            \includegraphics[scale = 0.3]{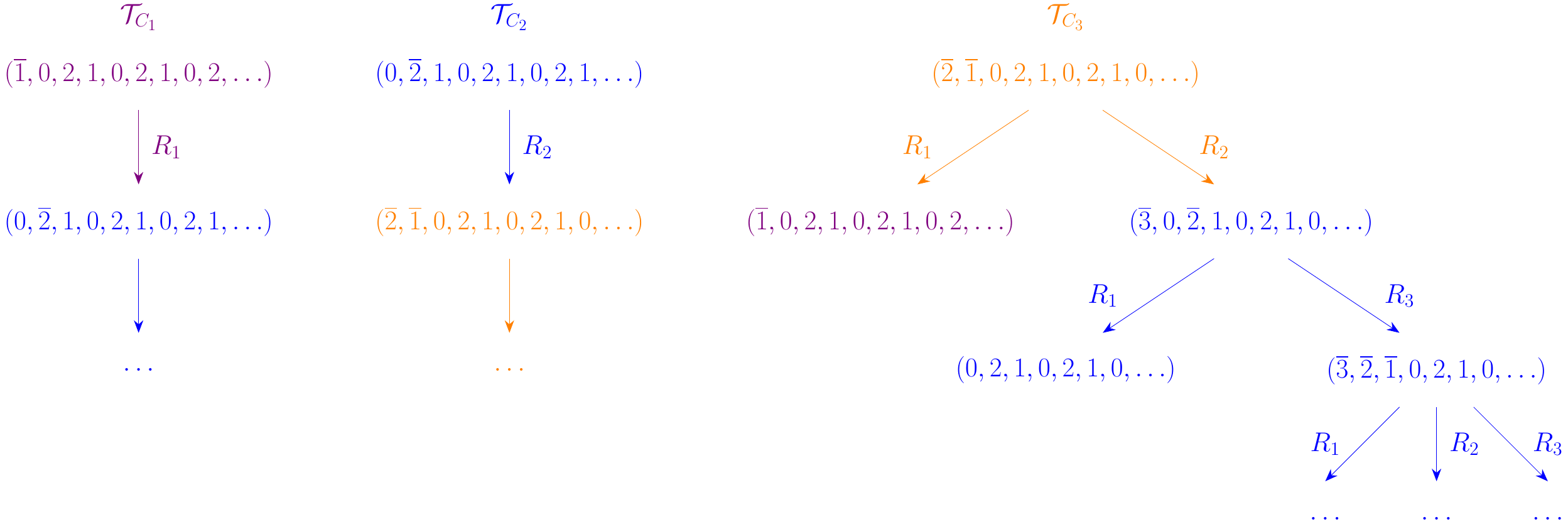}
            \caption{$\mathcal{F}_{WBW}$}
            \label{subfig:F_BWW_full}
        \end{subfigure}

        \vspace{2em}
     
        \begin{subfigure}[b]{\textwidth}
            \centering
            \includegraphics[scale = 0.3]{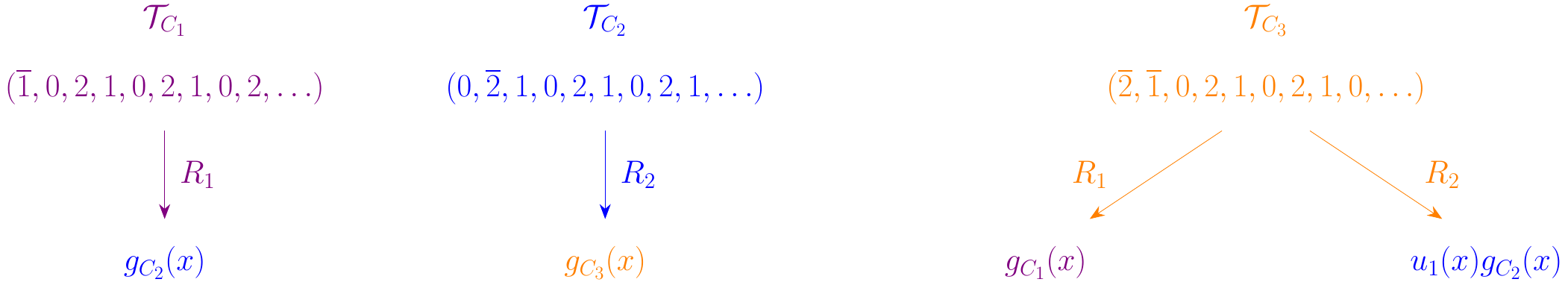}
            \caption{$\mathcal{F}_{WBW}$ degenerated}
            \label{subfig:F_BWW_degen}
        \end{subfigure}
        \caption{$\mathcal{F}_{WBW}$ and its degenerated version}
        \label{fig:F_BWW}
    \end{figure}

    Similarly, the recurrent set for $\mathcal{O}_{BWB}$ is
    \[ \{C_1,C_2,C_3\} = \{(\overline{2},0,1,2,0,1,\ldots),\quad (0,\overline{1},2,0,1,2,\ldots),\quad (\overline{1},\overline{2},0,1,2,0,\ldots)\}. \]
    The quasi-infinite forest $\mathcal{F}_{BWB}$ is shown in Figure \ref{subfig:F_BBW_full}, and its degenerated forest is shown in Figure \ref{subfig:F_BBW_degen}. One can easily check that the degenerated forests of $\mathcal{F}_{WBW}$ and $\mathcal{F}_{BWB}$ are isomorphic by comparing Figures \ref{subfig:F_BWW_degen} and \ref{subfig:F_BBW_degen}. Thus, $\mathcal{F}_{WBW}$ and $\mathcal{F}_{BWB}$ are almost isomorphic. Lemma \ref{lem:almost-iso} will show that this implies $\mathcal{F}_{WBW}$ and $\mathcal{F}_{BWB}$ are isomorphic, which can be seen by comparing Figures \ref{subfig:F_BWW_full} and \ref{subfig:F_BBW_full}. This will be our method for proving Theorem \ref{thm:BWB_WBW}.

    \begin{figure}[h!]
        \centering
        \begin{subfigure}[b]{\textwidth}
            \centering
            \includegraphics[scale = 0.3]{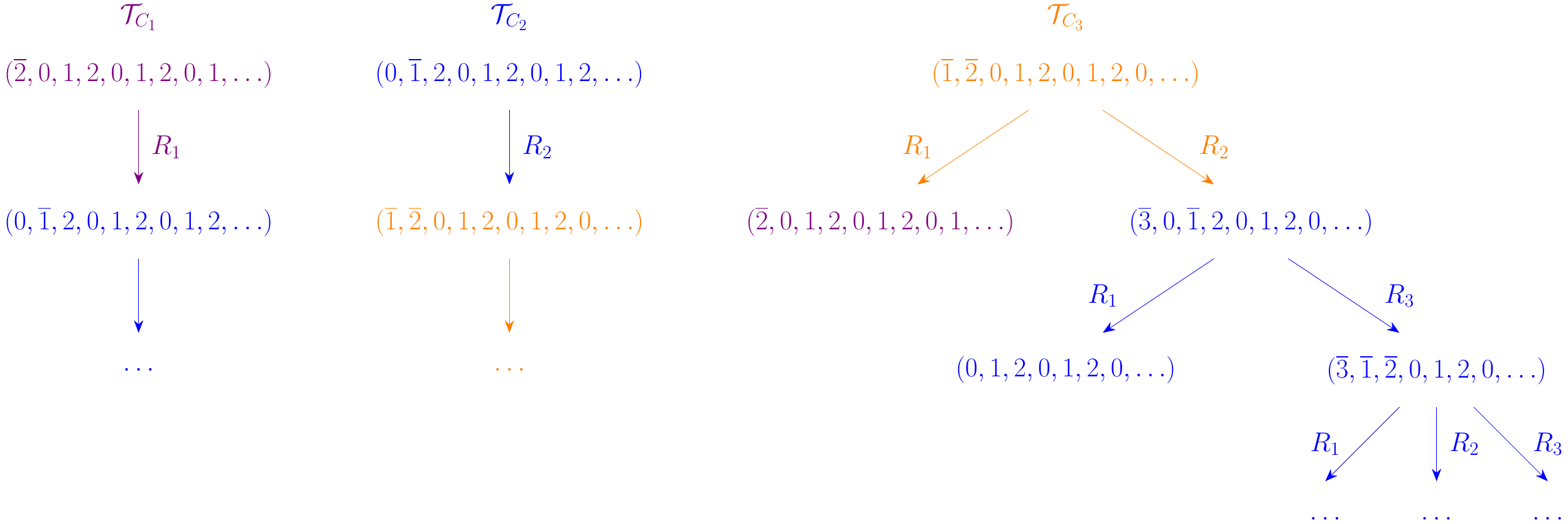}
            \caption{$\mathcal{F}_{BWB}$}
            \label{subfig:F_BBW_full}
        \end{subfigure}

        \vspace{2em}
     
        \begin{subfigure}[b]{\textwidth}
            \centering
            \includegraphics[scale = 0.3]{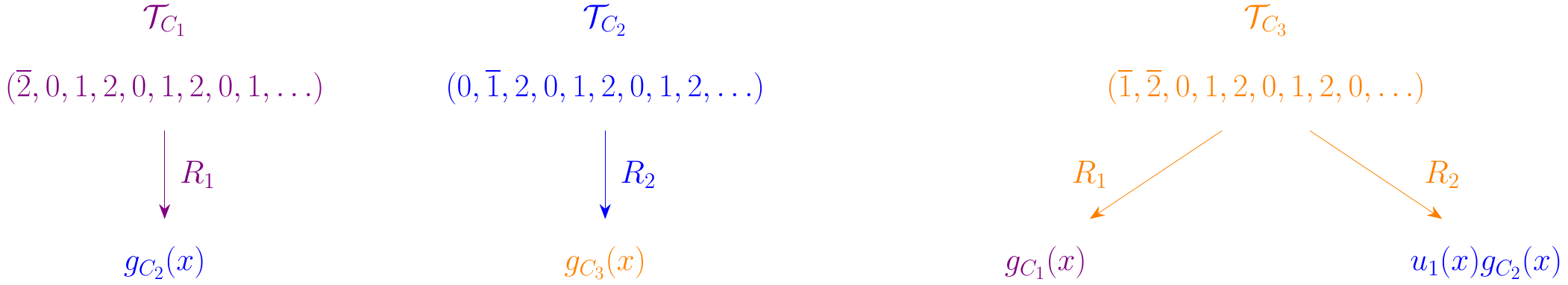}
            \caption{$\mathcal{F}_{BWB}$ degenerated}
            \label{subfig:F_BBW_degen}
        \end{subfigure}
        \caption{$\mathcal{F}_{BWB}$ and its degenerated version}
        \label{fig:F_BBW}
    \end{figure}

    \begin{lemma}\label{lem:almost-iso}
        If $\mathcal{T}_{C_i}$ and $\mathcal{T}_{C'_i}$ are almost isomorphic for all $i$, then $\mathcal{T}_{C_i}$ and $\mathcal{T}_{C'_i}$ are isomorphic for all $i$.
    \end{lemma}

    \begin{proof}
        It suffices to prove that from the roots of $\mathcal{T}_{C_i}$ and $\mathcal{T}_{C'_i}$, one can play a sequence of moves $R_{j_1},\ldots,R_{j_m}$ in $\mathcal{T}_{C_i}$ if and only if one can play the same sequence in $\mathcal{T}_{C'_i}$. Observe that we start at the root of both trees. If before $R_{j_r}$, we have non-degenerate elements in both trees, then $R_{j_r}$ is playable in one tree if and only if it is playable in the other. If we have degenerate elements in both trees, then by definition of almost isomorphic, the degenerate elements both have $k$-fuses followed by $C_j$ and $C_j'$ respectively. If $1\leq j_r\leq k$, then $R_{j_r}$ is playable in both trees, and playing $R_{j_r}$ leads to the ``terminal phase" of the $k$-fuses, which we already know are isomorphic. If $j_r>k$ then playing $R_{j_r}$ is the same as playing $R_{j_r-k}$ in $\mathcal{T}_{C_j}$ and $\mathcal{T}_{C'_j}$. Since $\mathcal{T}_{C_j}$ and $\mathcal{T}_{C'_j}$ are also almost isomorphic, $R_{j_r-k}$ is playable in one tree if and only if it is playable in the other.
    \end{proof}

    We first analyze the quasi-infinite trees of the families $B(WB)^k$ and $W(BW)^k$. For the former family, the recurrent cycle elements are
    \begin{align*}
        C_1 &= (\overline{2},0,2,0,2,0,\ldots, 2,0,1,\ldots)\\
        C_2 &= (0,\overline{2},0,2,0,\ldots, 2,0,1,2,\ldots)\\
        C_3 &= (\overline{2},0,\overline{2},0,\ldots, 2,0,1,2,0,\ldots)\\
        C_4 &= (0,\overline{2},0,\ldots, 2,0,1,2,0,2,\ldots)\\
        \vdots\\
        C_{2k-2} &= (0,\overline{2},0,1,2,0,2,\ldots, 0,2,\ldots)\\
        C_{2k-1} &= (\overline{2},0,\overline{1},2,0,2,\ldots, 0,2,0,\ldots)\\
        C_{2k} &= (0,\overline{1},2,0,2,\ldots, 0,2,0,2,\ldots)\\
        C_{2k+1} &= (\overline{1},\overline{2},0,2,\ldots, 0,2,0,2,0,\ldots)\\
    \end{align*}
    where each element has $k$ twos, $k$ zeros and a one. Similarly, the recurrent cycle elements of the latter family are
    \begin{align*}
        C'_1 &= (\overline{1},0,2,0,2,0,\ldots, 2,0,2,\ldots)\\
        C'_2 &= (0,\overline{2},0,2,0,\ldots, 2,0,2,1,\ldots)\\
        C'_3 &= (\overline{2},0,\overline{2},0,\ldots, 2,0,2,1,0,\ldots)\\
        C'_4 &= (0,\overline{2},0,\ldots, 2,0,2,1,0,2,\ldots)\\
        \vdots\\
        C'_{2k-2} &= (0,\overline{2},0,2,1,0,2,\ldots, 0,2,\ldots)\\
        C'_{2k-1} &= (\overline{2},0,\overline{2},1,0,2,\ldots, 0,2,0,\ldots)\\
        C'_{2k} &= (0,\overline{2},1,0,2,\ldots, 0,2,0,2,\ldots)\\
        C'_{2k+1} &= (\overline{2},\overline{1},0,2,\ldots, 0,2,0,2,0,\ldots)\\
    \end{align*}
    where each element also has $k$ twos, $k$ zeros and a one.

    Our first lemma is clear.

    \begin{lemma}\label{lem:1.1-trivial-case}
        For $i = 1$ and $i = 2j$ ($1\leq j\leq k$), $\mathcal{T}_{C_i}$ and $\mathcal{T}_{C'_i}$ are almost isomorphic.
    \end{lemma}

    \begin{proof}
        This is clear because for these $C_i$, there is only one playable move, which gives $C_{i+1}$. Thus, the degenerate tree has only two elements: $C_i$ and the degenerate element for $\mathcal{T}_{C_{i+1}}$.
    \end{proof}

    Our next lemma is also straightforward.

    \begin{lemma}\label{lem:1.1-easy-case}
        $\mathcal{T}_{C_{2k+1}}$ and $\mathcal{T}_{C'_{2k+1}}$ are almost isomorphic.
    \end{lemma}

    \begin{proof}
        From $C_{2k+1}$, we have two moves: $R_1$ and $R_2$. If we play $R_1$, we get the degenerate element for $\mathcal{T}_{C_1}$. If we play $R_2$, we get a $1$-fuse followed by $C_2$, which is also a degenerate element. The tree for $C'_{2k+1}$ is exactly the same, so they are almost isomorphic.
    \end{proof}

    Now we tackle the more complicated elements.

    \begin{lemma}\label{lem:1.1-general-case}
        For $i = 2j+1$ ($1\leq j\leq k-1$), $\mathcal{T}_{C_i}$ and $\mathcal{T}_{C'_i}$ are almost isomorphic.
    \end{lemma}

    \begin{proof}
        First note that
        \[ C_i = (\overline{2},0,\overline{2},0,\ldots,0,1,2,\ldots) \]
        and 
        \[ C'_i = (\overline{2},0,\overline{2},0,\ldots,0,2,1,\ldots). \]
        Specifically, both elements begin with $(\overline{2},0,\overline{2},0,2,0,\ldots)$ and the first difference is in the $(2k+1-2i)$th and $(2k+2-2i)$th parts where those of $C_i$ are $1,2$ while those of $C'_i$ are $2,1$. Let us call these two parts the \textbf{significant parts}. Until these two parts are played, the two trees are isomorphic. Now we claim that in order for the significant parts to be playable, we need to always play the last playable part, i.e. the playable part with the largest index. Indeed, suppose we have an element $(\overline{2},\overline{2},\ldots,\overline{2},0,\overline{2},0,\ldots)$ (note that $C_i$ and $C'_i$ also have this form themselves), if we play the last playable part, then we get another element of this form. If we do not play the last playable part, then we reach an $\ell$-fuse $(\overline{2},\ldots,\overline{4},\overline{2},2,\ldots,2,0,2,\ldots)$. From here, if we play $R_j$ with $j<\ell$, we trigger the terminating phase and will eventually stop before the significant parts are playable. Else, we can only play $R_\ell$ repeatedly until we get $(\overline{2},\ldots,\overline{2m},0,\overline{2},\ldots)$, which is an $\ell$-fuse followed by a recurrent cycle element. Thus, this degenerates to an element before the significant parts are playable.

        When the first significant part become playable, the elements in the two trees are $(\overline{2},\ldots,\overline{2},0,\overline{1},2,\ldots)$ and $(\overline{2},\ldots,\overline{2},0,\overline{2},1,\ldots)$. Similar to above, if we do not play the last playable part, we will either terminate or get a fuse followed by $C_{2k}$ and $C'_{2k}$, and so the subtrees are almost isomorphic. If we play the last playable part, then we get $(\overline{2},\ldots,\overline{2},\overline{1},\overline{2},0,2,\ldots)$ and $(\overline{2},\ldots,\overline{2},\overline{2},\overline{1},0,2,\ldots)$. Once again, if we do not play the last playable part, then the subtrees are almost isomorphic. If we play the last playable part, in both trees, we get $(\overline{2},\ldots,\overline{2},\overline{3},0,\overline{2},\ldots)$, which is a fuse followed by $C_2$ and $C'_2$, and so this degenerates to the same element in both trees, and hence the trees are almost isomorphic.
    \end{proof}

    The last three lemmas combine to prove Theorem \ref{thm:BWB_WBW}.

    \begin{repthm}{1.2}
        For $k \geq 1$, one has
        \[ H_{B(WB)^k}(x) = H_{W(BW)^k}(x) .\]
    \end{repthm}

    \begin{proof}
        From the lemmas, we have that $\mathcal{T}_{C_i}$ and $\mathcal{T}_{C'_i}$ are almost isomorphic for all $i$, so they are isomorphic for all $i$. Thus, the generating functions $g_i$ and $g'_i$ are the same for all $i$, and hence $H_{B(WB)^k} (x)= H_{W(BW)^k}(x)$.
    \end{proof}

\section{$BW^k$ and $WB^k$}\label{sec:thm-1.2}

    Now, we shift our focus to the families $BW^k$ and $WB^k$ and Theorem \ref{thm:BWW_WBB}, asserting that $H_{BW^k}(x)$ and $H_{WB^k}(x)$ can be both written as a rational function over the same denominator of degree $k+1$ for $k\geq 1$.

    Let us start with a warm-up example with $P=BWWW$. Figure \ref{fig:F_BWWW} shows the degenerated forest for this necklace.
    
    \begin{figure}[h!]
        \centering
        \includegraphics[scale = 0.5]{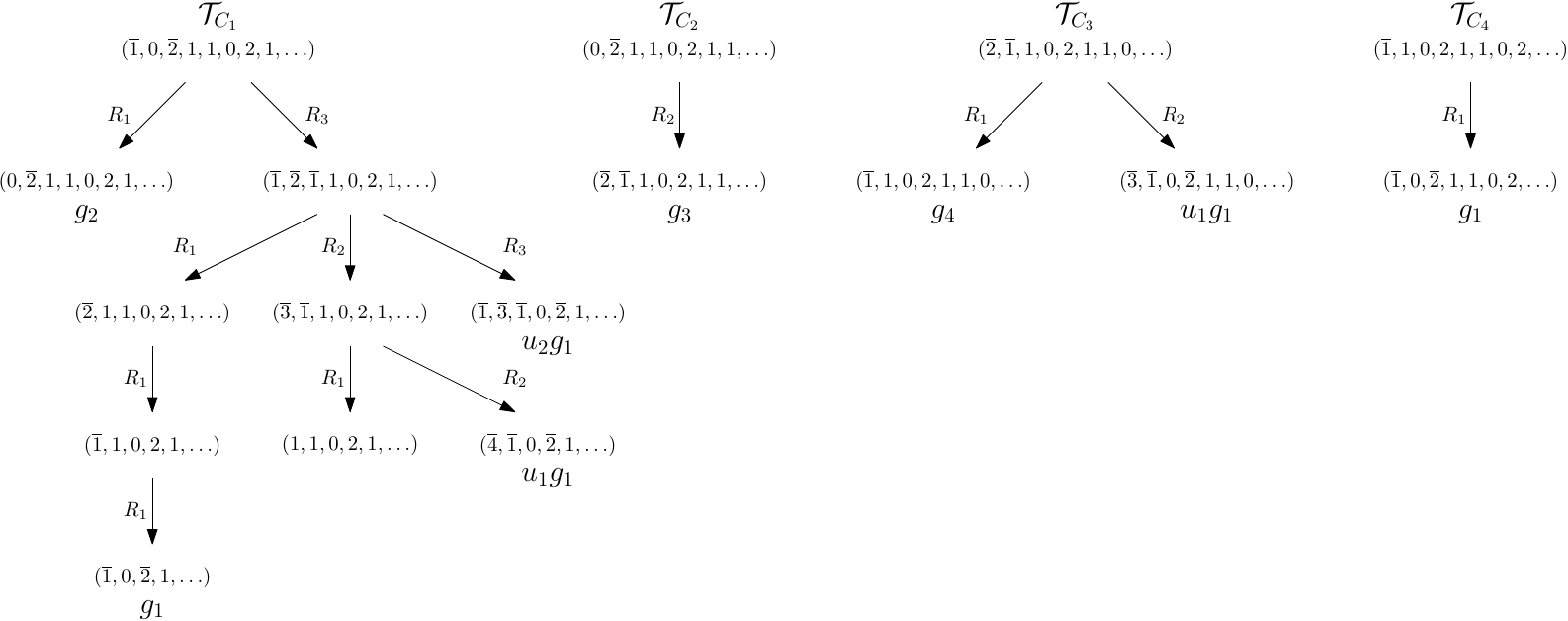}
        \caption{$F_{BWWW}$ degenerated}
        \label{fig:F_BWWW}
    \end{figure}

    From Figure \ref{fig:F_BWWW}, we can write the following system of equations
    \begin{align}\label{eqn:BWWW}
        \begin{cases}
            g_1 &= 1 + x + 2x^2 + 2x^3 + xg_2 +(x^4 + x^3u_1 + x^2u_2)g_1 \\
            g_2 &= 1 + xg_3\\
            g_3 &= 1 + xg_4 + xu_1g_1\\
            g_4 &= 1 + xg_1.
        \end{cases}
    \end{align}

    We can solve this system by substitution. We have
    \[ g_2 = 1 + xg_3 = 1 + x(1+xg_4 + xu_1g_1) =  1 + x(1+x(1+xg_1) + xu_1g_1) \]
    \[ = 1 + x + x^2 + x^3g_1 + x^2u_1g_1. \]
    Thus,
    \begin{align*}
        g_1 &= 1 + x + 2x^2 + 2x^3 + xg_2 + (x^4+x^3u_1+x^2u_2)g_1 \\
        &= A_1 + (2x^4+2x^3u_1+x^2u_2)g_1 \\
        &= A_1 + (2x^4 + 2x^3(1+x) + x^2(1+2x+2x^2))g_1 \\
        &= A_1 + (6x^4 + 4x^3 + x^2)g_1.
    \end{align*}
    Therefore,
    \[ g_1 = \dfrac{-A_1}{6x^4+4x^3+x^2-1}. \]
    Working backwards, we can solve for $g_2, g_3$ and $g_4$:
    \begin{align*}
        g_4 &= 1 + xg_1 = 1+ x\dfrac{-A_1}{6x^4+4x^3+x^2-1} = \dfrac{-A_4}{6x^4+4x^3+x^2-1} \\
        g_3 &= 1 + xg_4 + xu_1g_1 = \dfrac{-A_3}{6x^4+4x^3+x^2-1} \\
        g_2 &= 1 + xg_3 = \dfrac{-A_2}{6x^4+4x^3+x^2-1}
    \end{align*}
    for some polynomials $A_2,A_3,A_4$. Recall from Section \ref{subsec:quasi_forest} that this means
    \[ g = g_1+g_2+g_3+g_4 = \dfrac{-A_1(x)-A_2(x)-A_3(x)-A_4(x)}{6x^4+4x^3+x^2-1}, \]
    and hence
    \[ H_{BWWW}(x) = (1-x)\dfrac{-A_1(x)-A_2(x)-A_3(x)-A_4(x)}{6x^4+4x^3+x^2-1}. \]
    Thus, $H_{BWWW}(x)$ can be written as a rational generating function over a polynomial of degree $4$.

    Similarly, for $P = WBBB$, we have the following system of equations. We encourage the readers to check that this is the correct system.
    \begin{align*}
        \begin{cases}
            g_1 &= 1 + x + x^2 + xg_2 + (x^3 + x^2u_1 + xu_2)g_4 \\
            g_2 &= 1 + xg_3 + xu_1g_4\\
            g_3 &= 1 + xg_4\\
            g_4 &= 1 + xg_1.
        \end{cases}
    \end{align*}
    Although this system is a bit different from (\ref{eqn:BWWW}), by substitution, we also have
    \[ g_2 = 1 + xg_3 + xu_1g_4 = 1 + x(1 + x(1+xg_1)) + xu_1(1+xg_1) \]
    \[ = 1 + x + x^2 + xu_1 + x^3g_1 + x^2u_1g_1. \]
    Thus,
    \begin{align*}
        g_1 &= 1 + x + x^2 + xg_2 + (x^3 + x^2u_1 + xu_2)g_4 \\
        &= A'_1 + (x^4 + x^3u_1)g_1 + (x^4 + x^3u_1 + x^2u_2)g_1 \\
        &= A'_1 + (2x^4 + 2x^3u_1 + x^2u_2)g_1 \\
        &= A'_1 + (2x^4 + 2x^3(1+x) + x^2(1+2x+2x^2))g_1 \\
        &= A'_1 + (6x^4 + 4x^3 + x^2)g_1.
    \end{align*}
    This means that
    \[ g_1 = \dfrac{-A'_1}{6x^4 + 4x^3 + x^2-1}, \]
    and similar to above, eventually we have
    \[ H_{BWWW}(x) = (1-x)\dfrac{-A'_1(x)-A'_2(x)-A'_3(x)-A'_4(x)}{6x^4+4x^3+x^2-1} \]
    for some polynomial $A'_2,A'_3,A'_4$. This is also a generating function over the same polynomial of degree $4$ as $H_{BWWW}$.

    Observe that in both examples above, we use substitution to derive
    \[ g_1 = A + (6x^4+4x^3+x^2)g_1 \]
    for some polynomial $A$. This means we have
    \[ g_1 = \dfrac{-A}{6x^4+4x^3+x^2 -1}, \]
    and eventually we can write both $H_{BWWW}$ and $H_{WBBB}$ as a generating function over $6x^4+4x^3+x^2 -1$, which is a polynomial of degree $4$. This will be our main strategy in this section. 

    Also, observe from the example that in both cases, we encounter the sum $x^2 + xu_1 + u_2$. This sum is indeed homogeneous, for Proposition \ref{prop:u_n_formula} shows that $u_2(x)$ has degree $2$ and $u_1(x)$ has degree $1$. Thus, for our convenience, we will ``normalize" our $u_k$'s by redefining
    \[ u_k(x) = \sum_{i=0}^k c_{i,k-i}x^{i-k} \]
    where $c_{n,i}$ is the number of weak compositions of $n$ with exactly $i$ zeros. Under this new definition, the sum $x^2 + xu_1 + u_2$ becomes $x^2(u_0+u_1+u_2)$ (since $u_0 = 1$). This motivates the following abbreviation:

    \begin{definition}
        Let $v_k(x):= u_0(x)+u_1(x)+\ldots+u_k(x).$
    \end{definition}
    Now, we are ready to carry out the computations for Theorem \ref{thm:BWW_WBB}.

\subsection{$BW^k$}\label{subsec:bww}

    First, we will deal with the more ``friendly'' family of the two. The recurrent cycle elements of this family are
    \begin{align*}
        C_1 &= (\overline{1},\overline{1},1,1,1,\ldots,1,0,2,1,\ldots)\\
        C_2 &= (\overline{1},\overline{1},1,1,\ldots,1,0,2,1,1,\ldots)\\
        C_3 &= (\overline{1},\overline{1},1,\ldots,1,0,2,1,1,1,\ldots)\\
        \vdots \\
        C_{k-3} &= (\overline{1},\overline{1},0,2,1,1,\ldots,1,1,1,\ldots)\\
        C_{k-2} &= (\overline{1},0,\overline{2},1,1,\ldots,1,1,1,1,\ldots)\\
        C_{k-1} &= (0,\overline{2},1,1,\ldots,1,1,1,1,1,\ldots)\\
        C_{k} &= (\overline{2},\overline{1},1,\ldots,1,1,1,1,1,0,\ldots)\\
        C_{k+1} &= (\overline{1},1,\ldots,1,1,1,1,1,0,2,\ldots)
    \end{align*}

    \begin{definition}
    \label{first-definition-of-f's}
        For $k\geq 2$, define $f_k$ to be the polynomial such that when using substitution to solve the system of equations for $BW^k$, we have
        \[ g_1 = A + f_kg_1 \]
        for some polynomial $A$. Then $H_{BW^k}$ can be written as a rational generating function over $f_k -1 $.
    \end{definition}

    It is actually not clear yet why such $f_k$ always exists. Its existence will be proved in Proposition \ref{prop:BWW-recurrence}; furthermore, we will show that these $f_k$'s satisfy the recurrence given in ($\ref{equa:f-recurrence}$). Once we can write $g_1 = A + f_kg_1$, it follows that we can write $g_1$ as a rational generating function over $f_k -1 $. Along the way, Proposition \ref{prop:same-equation} implies that the equation for any $g_\ell$ only depends on $g_{\ell+1},\ldots,g_{k+1}$ and $g_1$. Thus, we can iteratively write $g_{k+1},\ldots,g_{2}$ as rational generating functions over $f_k - 1$. This implies that we can write $H_{BW^k}$ as a rational generating function over $f_k -1 $.
    
    Let us start once again with an example that will illustrate the idea of the recurrence. Figure \ref{fig:BWWWW_trees} shows the quasi-infinite trees corresponding to the recurrent cycle elements for $BWWWW=BW^4$. From the forest, we can set up the following system of equations
    \begin{align*}
        g_1 &= A_1^{(4)} + xv_0g_2 + x^2v_0g_3 + x^4v_1g_5 + x^5u_2g_1\\
        g_2 &= A_2^{(4)} + xv_0g_3 + x^3v_1g_5 + x^4u_2g_1\\
        g_3 &= A_3^{(4)} + xv_0g_4\\
        g_4 &= A_4^{(4)} + xv_0g_5 + x^2u_1g_1\\
        g_5 &= A_5^{(4)} + xv_0g_1
    \end{align*}
    where each $A_i^{(4)}$ is a sum of the terms whose weights do not contain any $g_i$. Hence, each $A_i^{(4)}$ is not relevant to our study of the denominator.

    \begin{figure}[h!]
        \centering
        \includegraphics[width = 0.8\textwidth]{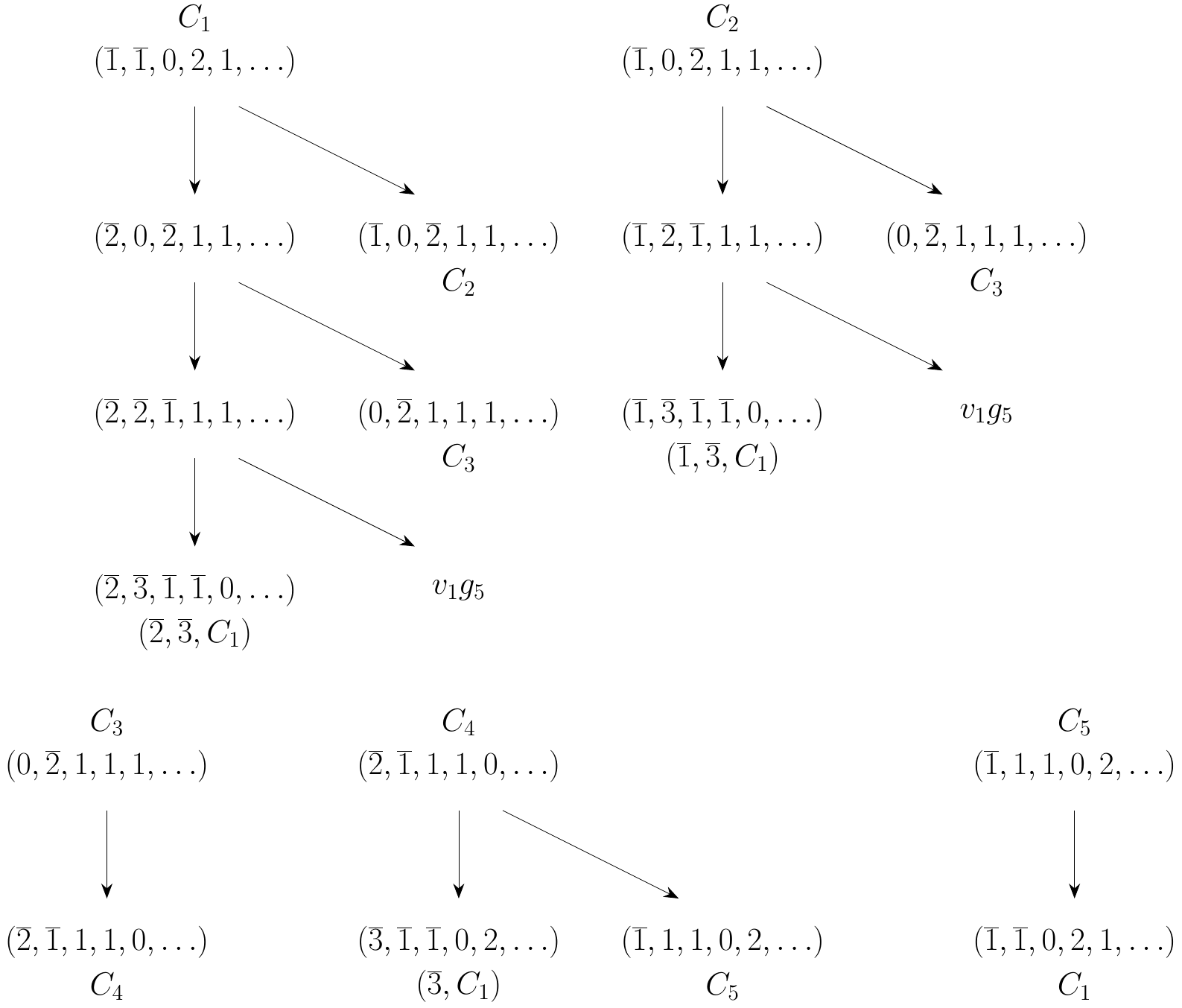}
        \caption{Quasi-infinite forest for $BWWWW$}
        \label{fig:BWWWW_trees}
    \end{figure}

    Similarly, we can set up the following system for $BWWWWW= BW^5$
    \begin{align*}
        g_1 &= A_1^{(5)} + xv_0g_2 + x^3v_1g_4 + x^5v_2g_6 + x^6u_3g_1 \\
        g_2 &= A_2^{(5)} + xv_0g_3 + x^2v_0g_4 + x^4v_1g_6 + x^5u_2g_1\\
        g_3 &= A_3^{(5)} + xv_0g_4 + x^3v_1g_6 + x^4u_2g_1\\
        g_4 &= A_4^{(5)} + xv_0g_5\\
        g_5 &= A_5^{(5)} + xv_0g_6 + x^2u_1g_1\\
        g_6 &= A_6^{(5)} + xv_0g_1
    \end{align*}
    Observe that the equations for $g_2,g_3,\ldots,g_6$ are exactly the same (up to shifting the indices) as those for $g_1,g_2,\ldots,g_5$ for $BW^4$. This is indeed true in general.

    \begin{prop}\label{prop:same-equation}
        If for $P=BW^k$ we have an equation
        \[ g_\ell = A_\ell^{(k)} + \sum_{i>0} x^ia_ig_{\ell+i} \]
        where $a_i$ is some coefficient (in this case $a_i$ is either $v_j$ or $u_j$ for some $j$), then for $P=BW^{k+m}$, we have an equation
        \[ g_{\ell+m} = A_{\ell+m}^{(k+m)} + \sum_{i>0} x^ia_ig_{\ell+m+i}. \]
    \end{prop}

    \begin{proof}
        The equation $g_\ell = A_\ell^{(k)} + \sum_{i>0} x^ia_ig_{\ell+i}$ for $P=BW^k$ is set up using the quasi-infinite tree rooted at the element
        \[ (\underbrace{\overline{1},\overline{1},1,\ldots,1}_{j\text{ copies of }1},0,2,1,\ldots) \]
        for some $j$. Then, for $P=BW^{k+m}$, the equation for $g_{\ell+m}$ is set up using the quasi-infinite tree rooted at the element of the exact same type, i.e.
        \[ (\underbrace{\overline{1},\overline{1},1,\ldots,1}_{j\text{ copies of }1},0,2,1,\ldots). \]
        Thus, the equations are the same up to shifting of the indices.
    \end{proof}

    Proposition \ref{prop:same-equation} leads to a useful corollary.

    \begin{cor}\label{cor:same-f}
        If for $P=BW^k$, we have
        \[ g_1 = A^{(k)} + f_kg_1 \]
        for some polynomial $A^{(k)}$, then for $P = BW^{k+m}$, we have
        \[ g_{m+1} = A^{(k+m)} + f_kg_1 \]
        for some polynomial $A^{(k+m)}$.
    \end{cor}

    \begin{proof}
        The equations for $g_1,\ldots,g_{k+1}$ of $BW^k$ is the same as those for $g_{m+1},\ldots,g_{k+m+1}$ for $BW^{k+m}$, so substitution yields the desired identity.
    \end{proof}

    For example, for $P=BW^4$, we computed that $g_1 = A^{(4)} + (12x^5+8x^4+2x^3)g_1$, then we know that for $P=BW^5$, we have $g_2 = A^{(5)} + (12x^5+8x^4+2x^3)g_1$. Thus, we obtain the following recurrence.

    \begin{prop}\label{prop:BWW-recurrence}
        The coefficients $f_n$ satisfy the following recurrence
        \begin{align}\label{equa:f-recurrence}
            f_n = \begin{cases}
                \left(\sum_{i = 0}^{\frac{n-4}{2}}x^{2i+1}v_if_{n-(2i+1)}\right) + x^{n-2}v_{\frac{n-4}{2}}f_2 + x^{n+1}v_{\frac{n}{2}} &~~\text{if $n$ is even}\\
                \left(\sum_{i = 0}^{\frac{n-3}{2}}x^{2i+1}v_if_{n-(2i+1)}\right) + x^{n+1}v_{\frac{n+1}{2}} &~~\text{if $n$ is odd}
            \end{cases}.
        \end{align}
    \end{prop}

    \begin{proof}
        This proposition is best illustrated by a figure.

        \begin{figure}[h!]
            \centering
            \includegraphics[width = \textwidth]{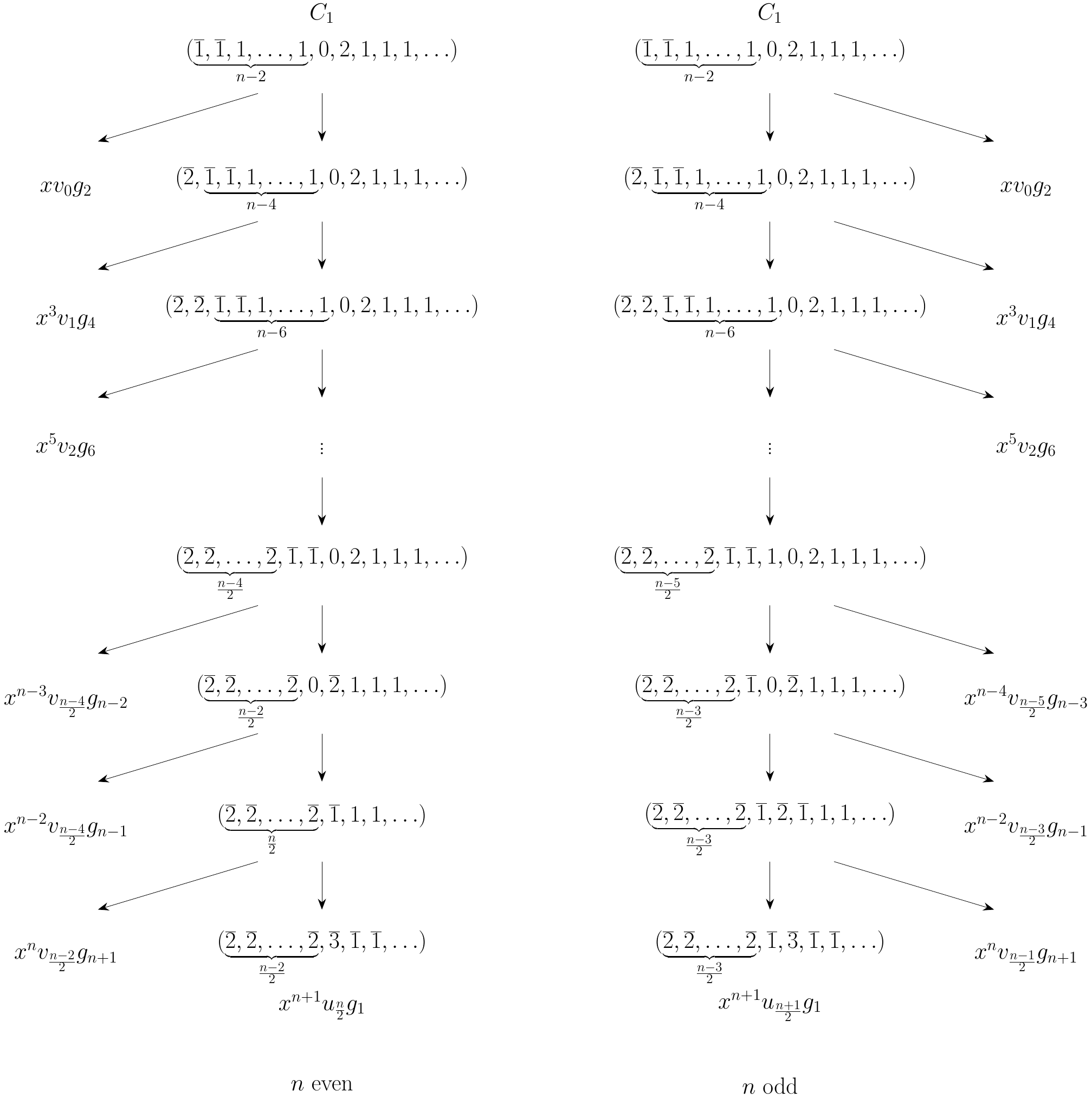}
            \caption{Quasi-infinite trees $\mathcal{T}_{C_1}$ when $n$ is even (left) and odd (right)}
            \label{fig:BWW_play_trees}
        \end{figure}

        Figure \ref{fig:BWW_play_trees} shows the quasi-infinite trees $\mathcal{T}_{C_1}$ for both cases. Hence, we can set up the equation
        \begin{align*}
            g_1 = A_1^{(n)} + \begin{cases}
                \left(\sum_{i = 0}^{\frac{n-4}{2}}x^{2i+1}v_ig_{2i+2}\right) + x^{n-2}v_{\frac{n-4}{2}}g_{n-1} + x^{n}v_{\frac{n-2}{2}}g_{n+1} + x^{n+1}u_{\frac{n}{2}}g_1 &\text{if $n$ is even}\\
                \left(\sum_{i = 0}^{\frac{n-3}{2}}x^{2i+1}v_ig_{2i+2}\right) + x^{n}v_{\frac{n-1}{2}}g_{n+1} + x^{n+1}u_{\frac{n+1}{2}}g_1 &\text{if $n$ is odd}
            \end{cases}.
        \end{align*}
        By Corollary \ref{cor:same-f}, each $g_{i+1}$ can be substituted by $A_{i+1} + f_{n-i}g_1$ for some polynomial $A_{i+1}$. In addition, in both cases, we have $g_{n+1} =  1 + xg_1$, and since $v_i + u_{i+1} = v_{i+1}$, we have
        \[ x^{n}v_{\frac{n-2}{2}}g_{n+1} + x^{n+1}u_{\frac{n}{2}}g_1 = x^{n}v_{\frac{n-2}{2}} + x^{n+1}v_{\frac{n}{2}}g_1 \]
        for even $n$ and
        \[ x^{n}v_{\frac{n-1}{2}}g_{n+1} + x^{n+1}u_{\frac{n+1}{2}}g_1 = x^{n}v_{\frac{n-1}{2}} + x^{n+1}v_{\frac{n+1}{2}}g_1 \]
        for odd $n$. Thus,
        \begin{align*}
            g_1 = A^{(n)} + \begin{cases}
                \left(\sum_{i = 0}^{\frac{n-4}{2}}x^{2i+1}v_if_{n-(2i+1)}g_1\right) + x^{n-2}v_{\frac{n-4}{2}}f_2g_1 + x^{n+1}v_{\frac{n}{2}}g_1 &~~\text{if $n$ is even}\\
                \left(\sum_{i = 0}^{\frac{n-3}{2}}x^{2i+1}v_ig_{2i+2}\right) + x^{n+1}v_{\frac{n+1}{2}}g_1 &~~\text{if $n$ is odd}
            \end{cases}
        \end{align*}
        for some polynomial $A^{(n)}$. This gives equation \eqref{equa:f-recurrence}.
    \end{proof}

    \begin{cor}\label{cor:deg_BWWW}
        For all $n$, $f_n$ has degree $n+1$.
    \end{cor}

    \begin{proof}
        This is immediate from (\ref{equa:f-recurrence}), knowing that $v_i$ has degree $0$ for all $i$.
    \end{proof}

    Corollary \ref{cor:deg_BWWW} means that for all $n$, $H_{BW^n}$ can be written as a generating function over $f_n-1$, which is a polynomial of degree $n+1$.

\subsection{$WB^k$}\label{subsec:WBB}

    Now we will shift our attention to the other family, namely $WB^k$. The recurrent cycle elements of this family are
    \begin{align*}
        C_1 &= (\overline{1},\overline{1},\overline{1},1,1,\ldots,2,0,\ldots)\\
        C_2 &= (\overline{1},\overline{1},1,1,\ldots,2,0,1,\ldots)\\
        C_3 &= (\overline{1},\overline{1},1,\ldots,2,0,1,1,\ldots)\\
        \vdots \\
        C_{k-2} &= (\overline{1},\overline{1},2,0,1,1,1,\ldots,1,\ldots)\\
        C_{k-1} &= (\overline{1},\overline{2},0,1,1,1,\ldots,1,1,\ldots)\\
        C_{k} &= (\overline{2},0,1,1,1,\ldots,1,1,1,\ldots)\\
        C_{k+1} &= (0,\overline{1},1,1,\ldots,1,1,1,2,\ldots)
    \end{align*}
    Figure \ref{fig:WBBBB_trees} shows the forest for $P = WBBBB = WB^4$.
    
    \begin{figure}[h!]
        \centering
        \includegraphics[width = \textwidth]{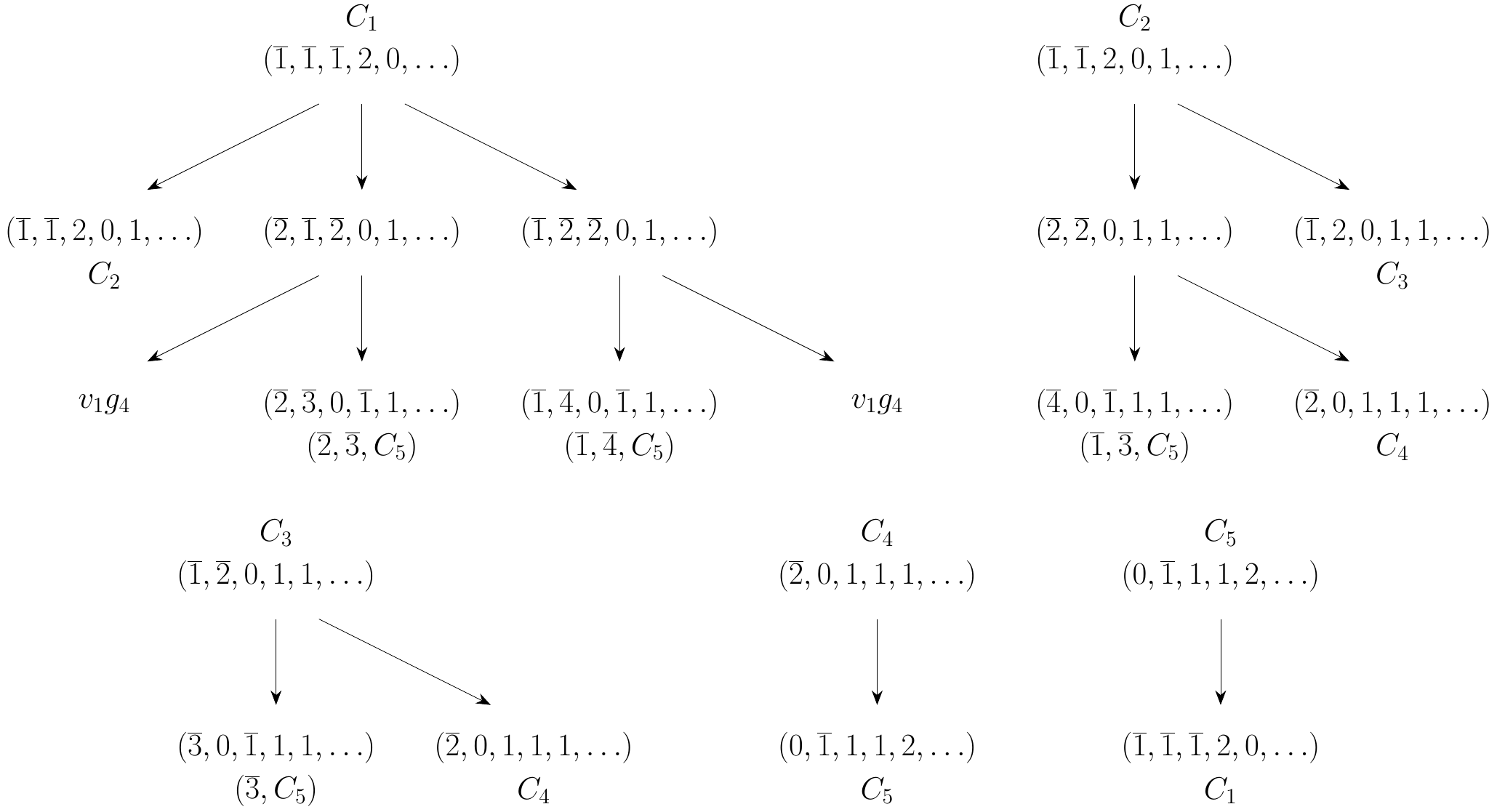}
        \caption{Quasi-infinite forest for $WBBBB$}
        \label{fig:WBBBB_trees}
    \end{figure}

    \noindent From Figure \ref{fig:WBBBB_trees}, we can set up the following system of equations.
    \begin{align*}
        g_1 &= B_{1}^{(4)} + xv_0g_2 + x^3v_1g_4 + x^4u_2g_5 + x^3v_1g_4 + x^4u_2g_5\\
        g_2 &= B_{2}^{(4)} + xv_0g_3 + x^2v_0g_4 + x^3u_1g_5\\
        g_3 &= B_{3}^{(4)} + xv_0g_4 + x^2u_1g_5\\
        g_4 &= B_{4}^{(4)} + xv_0g_5\\
        g_5 &= B_{5}^{(4)} + xv_0g_1
    \end{align*}
    Similarly, we can set up the following system for $P = WB^5$.
    \begin{align*}
        g_1 &= B_{1}^{(5)} + xv_0g_2 + x^3v_1g_4 + x^4v_1g_5 + x^5u_2g_6 + x^4v_2g_5 + x^5u_3g_6\\
        g_2 &= B_{2}^{(5)} + xv_0g_3 + x^3v_1g_5 + x^4u_2g_6\\
        g_3 &= B_{3}^{(5)} + xv_0g_4 + x^2v_0g_5 + x^3u_1g_6\\
        g_4 &= B_{4}^{(5)} + xv_0g_5 + x^2u_1g_6\\
        g_5 &= B_{5}^{(5)} + xv_0g_6\\
        g_6 &= B_{6}^{(5)} + xv_0g_1
    \end{align*}
    Similar to the case for $P = BW^k$, we can see that the equations for $g_2,\ldots,g_5$ for $WB^4$ are the same as those for $g_3,\ldots,g_6$ for $WB^5$. However, there is a \textit{minor difference} between the equation for $g_1$ for $WB^4$ and that for $g_2$ for $WB^5$. This is because $C_1$ has three playable parts, so besides the \textit{main branch} after playing $R_1$ and $R_2$, we also have the \textit{extra branch} after playing $R_3$. This extra branch, however, does not show up in longer necklaces, so the equations are different. To take into account this minor difference, we have a slightly different definition.

    \begin{definition}
        For $k\geq 2$, let $h_k$ be the polynomial such that when using substitution to solve the system of equation for $WB^{k+1}$, one has
        \[ g_2 = B + h_k g_1 \]
        for some polynomial $B$.
    \end{definition}

    \noindent Note that we have to define $h_k$ using the tree $\mathcal{T}_{C_2}$ of $WB^{n+1}$ to account for the minor difference above. We also have two results analogous to Proposition \ref{prop:same-equation} and Corollary \ref{cor:same-f}.

    \begin{prop}\label{prop:same-equation-WBB}
        If for $P=WB^k$ we have an equation
        \[ g_\ell = B_{\ell}^{(k)} + \sum_{i>0} x^ia_ig_{\ell+i} \]
        where $\ell > 1$ and $a_i$ is some coefficient (in this case $a_i$ is either $v_j$ or $u_j$ for some $j$), then for $P=BW^{k+m}$, we have an equation
        \[ g_{\ell+m} = B_{\ell+m}^{(k+m)} +  \sum_{i>0} x^ia_ig_{\ell+m+i}. \]
    \end{prop}

    \begin{cor}\label{cor:same-h}
        If for $P=BW^{k+1}$, we have
        \[ g_2 = B^{(k)} + h_kg_1 \]
        for some polynomial $B^{(k)}$, then for $P = BW^{k+m+1}$, we have
        \[ g_{m+1} = B^{(k+m)} + h_kg_1 \]
        for some polynomial $B^{(k+m)}$.
    \end{cor}

    \noindent Moreover, we also have an analogous recurrence.

    \begin{prop}\label{prop:WBB-recurrence}
        The coefficients $h_n$ satisfy the following recurrence
        \begin{align}\label{equa:h-recurrence}
            h_n = B + \begin{cases}
                \left(\sum_{i = 0}^{\frac{n-4}{2}}x^{2i+1}v_ih_{n-(2i+1)}\right) +  x^{n+1}v_{\frac{n}{2}} &~~\text{if $n$ is even}\\
                \left(\sum_{i = 0}^{\frac{n-3}{2}}x^{2i+1}v_ih_{n-(2i+1)}\right) + x^{n+1}v_{\frac{n-1}{2}} &~~\text{if $n$ is odd}
            \end{cases}
        \end{align}
        for some polynomial $B$.
    \end{prop}
    \begin{proof}

        Once again, this proposition is best illustrated by a figure.
        \begin{figure}[h!]
            \centering
            \includegraphics[width = \textwidth]{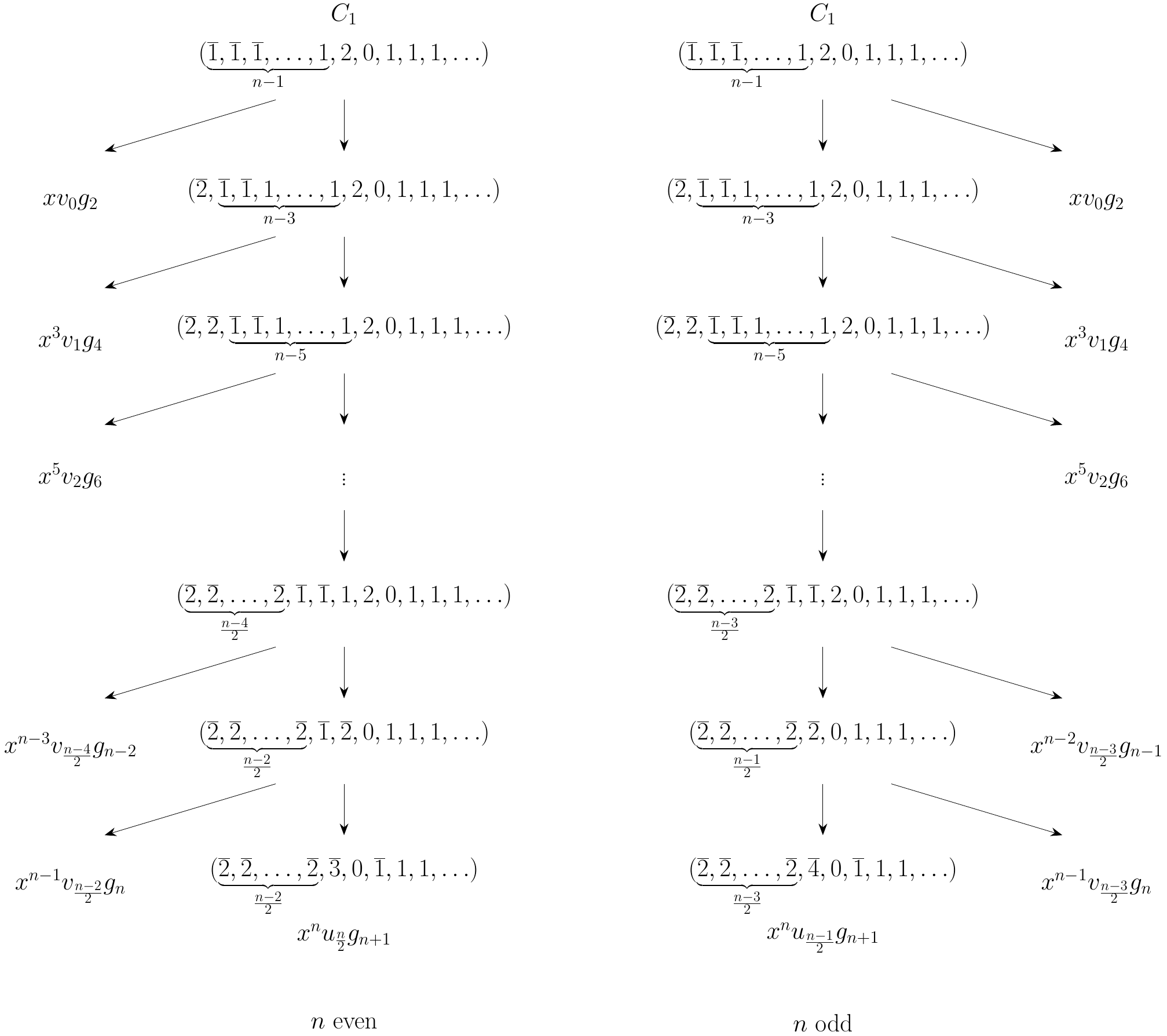}
            \caption{Main branch of $\mathcal{T}_{C_1}$ when $n$ is even (left) and odd (right)}
            \label{fig:WBB_play_trees}
        \end{figure}
    Figure \ref{fig:WBB_play_trees} shows the main branch of $\mathcal{T}_{C_1}$ for both cases. Note that in both cases, we have
        \[ g_n = 1 + xg_{n+1} = 1 + x + x^2g_1 \]
        and since $v_{i} + u_{i+1} = v_{i+1}$, we have
        \[ x^{n-1}v_{\frac{n-2}{2}}g_n + x^nu_{\frac{n}{2}}g_{n+1} = C + x^{n+1}v_{\frac{n}{2}}g_1 \]
        when $n$ is even, and
        \[ x^{n-1}v_{\frac{n-3}{2}}g_n + x^nu_{\frac{n-1}{2}}g_{n+1} = C + x^{n+1}v_{\frac{n-1}{2}}g_1 \]
        when $n$ is odd for some polynomial $C$. This gives equation \ref{equa:h-recurrence}.
    \end{proof}

    \begin{cor}\label{cor:deg_h_n}
        For all $n$, $h_n$ has degree $n+1$.
    \end{cor}

    \begin{proof}
        This is immediate from equation (\ref{equa:h-recurrence}), knowing that $v_i$ has degree $0$ for all $i$.
    \end{proof}

    Before wrapping up this subsection, let us give the relationship between the coefficients $h_n$ and the denominators of $H_{WB^n}(x)$.

    \begin{prop}\label{prop:denom-WBB-h}
        For $n\geq 4$, define $p_n(x)$ by the condition that when using substitution to solve the system of equations for $WB^n$, one has
        \[ g_1 = A + p_ng_1 \]
        for some polynomial $A$. Thus, $H_{WB^n}$ can be written as a generating function over $p_n-1$. Then,
        \[ p_n = x^{-1}h_{n+1} - x^2v_1h_{n-2}. \]
    \end{prop}

    \begin{proof}
        Note that $p_n$ satisfies
        \[ g_1 = B + p_ng_1 \]
        for some polynomial $B$ when solving the system of equations for $WB^n$. Let us compute $p_n$. The main branch of $C_1$ contributes $h_1$ to $p_n$. As for the extra branch, observe that the elements of this branch has the form $(\overline{1},\overline{2},\overline{2},\ldots)$. If we replace the first part $\overline{1}$ by $\overline{2}$, we get the elements in the main branch of $C_1$ in $WB^{n+1}$. Note that this replacement does not change the weight since it does not change the pre-fuses and fuses. Thus, the elements in the extra branch of $C_1$ in $WB^n$ are those in the main branch of $C_1$ in $WB^{n+1}$, with a few top elements missing, namely $xv_0g_2$ and $x^3v_1g_4$ (both in $WB^{n+1}$). Thus, the extra branch contributes $x^{-1}(h_{n+1}-xv_0h_n - x^3v_1h_{n-2})$, where the coefficient $x^{-1}$ is needed to shift the exponents. Since $v_0 = 1$, we have
        \[ p_n = h_n + x^{-1}(h_{n+1}-xv_0h_n - x^3v_1h_{n-2}) = x^{-1}h_{n+1} - x^2v_1h_{n-2}. \]
    \end{proof}

    \begin{cor}\label{cor:deg_WBBB}
        For all $n$, $p_n$ has degree $n+1$.
    \end{cor}

    \begin{proof}
        This is immediate from the equation
        \[ p_n = x^{-1}h_{n+1} - x^2v_1h_{n-2} \]
        because $h_{n+1}$ has degree $n+2$ and $h_{n-2}$ has degree $n-1$ (by Corollary \ref{cor:deg_h_n}).
    \end{proof}

    By Corollary \ref{cor:deg_WBBB}, for all $n$, $H_{WB^n}$ can be written as a generating function over $p_n-1$, which is a polynomial of degree $n+1$.

\subsection{Proof of Theorem \ref{thm:BWW_WBB}}\label{subsec:proof-1.2}

    Now we are ready to prove Theorem \ref{thm:BWW_WBB}.

    \begin{repthm}{1.3}
        For all $k \geq 1$, the functions $H_{BW^k}(x)$ and $H_{WB^k}(x)$ can both be written over the same denominator which is a polynomial of degree $k+1$.
    \end{repthm}

    \begin{proof}

        By Corollary \ref{cor:deg_BWWW}, $H_{BW^k}$ can be written as a generating function over $f_k-1$, which is a polynomial of degree $k+1$. By Corollary \ref{cor:deg_WBBB}, $H_{WB^k}$ can be written as a generating function over $p_k-1$, which is also a polynomial of degree $k+1$. Hence, it suffices to prove that $f_k = p_k$ for all $k$.

        For the base cases, when $k = 1$, $BW$ and $WB$ are the same necklace, so $H_{BW} = H_{WB}$. For $k = 2$, Pham in \cite{pham2022limiting} (and Theorem \ref{thm:BWB_WBW}) showed that $H_{BWW} = H_{WBB}$. For $k = 3$, the example at the beginning of this section showed that $H_{BWWW}$ and $H_{WBBB}$ can both be written as a generating function over $6x^4 + 4x^3 + x^2 -1$, which is a polynomial of degree $4$. In particular, $p_k = f_k$ for $k \leq 3$. Thus, it suffices to prove that $p_k$ satisfies equation \eqref{equa:f-recurrence} for $k\geq 4$.

        \begin{itemize}
            \item Case 1: $k$ is even. We need to check
            \[ p_k = \left(\sum_{i = 0}^{\frac{k-4}{2}}x^{2i+1}v_ip_{k-(2i+1)}\right) + x^{k-2}v_{\frac{k-4}{2}}p_2 + x^{k+1}v_{\frac{k}{2}}. \]
            Substituting $p_i = x^{-1}h_{i+1} - x^2v_1h_{i-2}$ for $i\geq 4$, this is equivalent to
            \begin{align*}
                x^{-1}h_{k+1} - x^2v_1h_{k-2} &= \left(\sum_{i = 0}^{\frac{k-6}{2}}x^{2i+1}v_i \left(x^{-1}h_{k-2i} - x^2v_1h_{k-2i-3}\right) \right)\\
                &+x^{k-3}v_{\frac{k-4}{2}}p_3 + x^{k-2}v_{\frac{k-4}{2}}p_2 + x^{k+1}v_{\frac{k}{2}}.
            \end{align*}
            From equation \ref{equa:h-recurrence}, we have
            \[ h_{k+1} = \left(\sum_{i = 0}^{\frac{k-6}{2}}x^{2i+1}v_ih_{k-2i}\right) + x^{k-3}v_{\frac{k-4}{2}}h_4 + x^{k-1}v_{\frac{k-2}{2}}h_2 + x^{k+2}v_{\frac{k}{2}} \]
            and
            \[ h_{k-2} = \left(\sum_{i = 0}^{\frac{k-6}{2}}x^{2i+1}v_ih_{k-2i-3}\right) +  x^{k-1}v_{\frac{k-2}{2}}. \]
            Thus, it suffices to check
            \[ x^{-1}\left( x^{k-3}v_{\frac{k-4}{2}}h_4 + x^{k-1}v_{\frac{k-2}{2}}h_2 + x^{k+2}v_{\frac{k}{2}} \right) - x^2v_1\cdot x^{k-1}v_{\frac{k-2}{2}} \]
            \[ = x^{k-3}v_{\frac{k-4}{2}}p_3 + x^{k-2}v_{\frac{k-4}{2}}p_2 + x^{k+1}v_{\frac{k}{2}}. \]
            
            \noindent Fortunately, this can be checked by direct computation. We have $h_4 = x^5(2v_1 + v_2)$, $h_2 = x^3v_1$, $p_3 = x^4(v_1+v_2)$, and $p_2 = x^3v_1$. Hence,
            \begin{align*}
                \text{LHS} &= x^{k+1}v_{\frac{k-4}{2}}(2v_1 + v_2) + x^{k+1}v_{\frac{k-2}{2}}v_1 + x^{k+1}v_{\frac{k}{2}} - x^{k+1}v_{\frac{k-2}{2}}v_1\\
                &= x^{k+1}v_{\frac{k-4}{2}}(v_1 + v_2) + x^{k+1}v_{\frac{k-4}{2}}v_1 + x^{k+1}v_{\frac{k}{2}}\\
                &= \text{RHS}.
            \end{align*}

            \item Case 2: $k$ is odd. We need to check
            \[ p_k = \left(\sum_{i = 0}^{\frac{k-3}{2}}x^{2i+1}v_ip_{k-(2i+1)}\right) + x^{k+1}v_{\frac{k+1}{2}} \]
            Substituting $p_i = x^{-1}h_{i+1} - x^2v_1h_{i-2}$ for $i\geq 4$, this is equivalent to
            \begin{align*}
                x^{-1}h_{k+1} - x^2v_1h_{k-2} &= \left(\sum_{i = 0}^{\frac{k-5}{2}}x^{2i+1}v_i \left(x^{-1}h_{k-2i} - x^2v_1h_{k-2i-3}\right) \right)\\
                &+x^{k-2}v_{\frac{k-3}{2}}p_2  + x^{k+1}v_{\frac{k+1}{2}}.
            \end{align*}
            From equation \ref{equa:h-recurrence}, we have
            \[ h_{k+1} = \left(\sum_{i = 0}^{\frac{k-5}{2}}x^{2i+1}v_ih_{k-2i}\right) + x^{k-2}v_{\frac{k-3}{2}}h_3 + x^{k+2}v_{\frac{k+1}{2}} \]
            and
            \[ h_{k-2} = \left(\sum_{i = 0}^{\frac{k-5}{2}}x^{2i+1}v_ih_{k-2i-3}\right) +  x^{k-1}v_{\frac{k-3}{2}}. \]
            Thus, it suffices to check
            \[ x^{-1}\left( x^{k-2}v_{\frac{k-3}{2}}h_3 + x^{k+2}v_{\frac{k+1}{2}} \right) - x^2v_1\cdot x^{k-1}v_{\frac{k-3}{2}} \]
            \[ = x^{k-2}v_{\frac{k-3}{2}}p_2  + x^{k+1}v_{\frac{k+1}{2}}. \]
            Again, by manual computation, we have $h_3 = 2x^4v_1$ and $p_2 = x^3v_1$. Hence,
            \begin{align*}
                \text{LHS} &= 2x^{k+1}v_{\frac{k-3}{2}}v_1 + x^{k+1}v_{\frac{k+1}{2}} - x^{k+1}v_{\frac{k-3}{2}}v_1\\
                &= x^{k+1}v_{\frac{k-3}{2}}v_1 + x^{k+1}v_{\frac{k+1}{2}}\\
                &= \text{RHS}.
            \end{align*}
        \end{itemize}

        Therefore, $p_k$ satisfies equation \eqref{equa:f-recurrence}, so the proof is complete.
    \end{proof}

\section{Discussion}\label{sec:discussion}

    As mentioned in the introduction, Theorem \ref{thm:BWW_WBB} is a special case of Conjecture \ref{con:P-and-P-dual} 
    on the duality operation for primitive necklaces $P \mapsto P^*$, since the dual of $BW^k$ is $B^kW = WB^k$. Theorem \ref{thm:BWB_WBW} is also a special case of this conjecture, but the two families $B(WB)^k$ and $W(BW)^k$ also have a stronger property that $H_{B(WB)^k(x)}$ and $H_{W(BW)^k}(x)$ are the same. We hope that our new representation and the combinatorial interpretation of $k$-fuses may lead to a proof of the conjecture. Furthermore, our proof of Theorem \ref{thm:BWW_WBB} is computationally heavy and is not combinatorial, so a new combinatorial proof of Theorem \ref{thm:BWW_WBB} may shed light on a proof of the general conjecture.

    Pham proposed another nice conjecture about the size of the {\it finite} Bulgarian solitaire orbits $\mathcal{O}_{P^k}$ for primitive necklaces $P$.

    \begin{conjecture}
        For any primitive necklace $P$ with $|P| \geq 3$, there is an integer $c_P$ such that for all $k$,
        \[ |\mathcal{O}_{P^k}| = c_P^{k-1}|\mathcal{O}_{P}|. \]
    \end{conjecture}

    If such $c_P$ exist, then there is an even more beautiful conjecture.

    \begin{conjecture}
        For any primitive necklace $P$ such that $c_{P}$ and $c_{P^*}$ both exist,
        \[ c_P = c_{P^*}. \]
    \end{conjecture}

    A special case was proved in by Pham in her thesis.

    \begin{thm}
        For all $k$,
        \[ |\mathcal{O}_{(BWW)^k}| = 5^k \]
        and
        \[ |\mathcal{O}_{(BBW)^k}| = 7\cdot5^{k-1}. \]
        Thus,
        \[ c_{BWW} = c_{BBW} = 5. \]
    \end{thm}

    It is also an interesting question to find a combinatorial interpretation of these $c_P$. In addition, the relationship between $c_P$ and the denominator of $H_P$ is not clear. For example, $BWBWBWB$ and $BWBBWWW$ are not the dual of each other, and $c_{BWBWBWB} = 63 \neq c_{BWBBWWW} = 94$, yet our data shows that
    \[ H_{BWBWBWB}(x) = H_{BWBBWWW}(x) \]
    \[ = (1-x)\dfrac{x^9 + 8x^8 + 42x^7 - 19x^6 - 63x^5 - 56x^4 - 34x^3 - 18x^2 - 10x - 7}{18x^7 + 16x^6 + 6x^5 + x^4 - 1}.\]
    The converse appears to be more probable. The smallest and only interesting example that we could compute is $WWWBBWWB$ and $WWWBBWBB$. They are not the dual of each other, but our data shows that
    \[ c_{WWWBBWWB} = c_{WWWBBWBB} = 135, \]
    and indeed $H_{WWWBBWWB}(x)$ and $H_{WWWBBWBB}(x)$ have the same denominator. Thus, we make the following conjecture.

    \begin{conjecture}
        For any two primitive necklaces $P_1$ and $P_2$, if $c_{P_1} = c_{P_2}$ then $H_{P_1}(x)$ and $H_{P_2}(x)$ have the same denominator.
    \end{conjecture}
    
    More data about $c_P$ and $H_P$ can be found in the Appendix below.

%--------------------------------------------------

\section*{Acknowledgments}
We would like to offer our sincerest thanks to Nhung Pham for sparking our interest in the Bulgarian Solitaire Problem, as well as to Vic Reiner for providing us with his expert guidance throughout this project. We would also like to thank Elise Catania and Connor McCausland for their help with editing and proofreading.

\bibliography{bibliography}
\bibliographystyle{alpha}

\newpage

\section*{Appendix: Data}
\subsection*{Data on the conjectural ratios $c_P$}

    These tables show the conjectural $c_P$ and $|\mathcal{O}_{P^k}|$ for primitive necklaces of size up to 8.

    \begin{table}[h!]
        \centering
        \begin{tabular}{|c|c|c|c|}
            \hline
            $P$ & $c_P$ & $|\mathcal{O}_{P^k}|$ & Verified for \\
            \hline
            $BWWW$ & 15 & $15\cdot 15^{k-1}$ & $k\leq 6$ \\
            \hline
            $BBBW$ & 15 & $30\cdot 15^{k-1}$ & $k\leq 6$ \\
            \hline
            \hline
            $BBWW$ & 10 & $15\cdot 10^{k-1}$ & $k\leq 6$ \\
            \hline
        \end{tabular}
        \caption{$|\mathcal{O}_{P^k}|$ and $c_P$ for primitive necklaces of size $4$}
        \label{tab:c_P_4}
    \end{table}
    
    \begin{table}[h!]
        \centering
        \begin{tabular}{|c|c|c|c|}
            \hline
            $P$ & $c_P$ & $|\mathcal{O}_{P^k}|$ & Verified for \\
            \hline
            $BWWWW$ & 44 & $56\cdot 44^{k-1}$ & $k\leq 4$ \\
            \hline
            $BBBBW$ & 44 & $135\cdot 44^{k-1}$ & $k\leq 4$ \\
            \hline
            \hline
            $BBWWW$ & 27 & $45\cdot 27^{k-1}$ & $k\leq 4$ \\
            \hline
            $BBBWW$ & 27 & $67\cdot 27^{k-1}$ & $k\leq 4$ \\
            \hline
            \hline
            $BWBWB$ & 17 & $34\cdot 17^{k-1}$ & $k\leq 5$ \\
            \hline
            $WBWBW$ & 17 & $32\cdot 17^{k-1}$ & $k\leq 5$ \\
            \hline
        \end{tabular}
        \caption{$|\mathcal{O}_{P^k}|$ and $c_P$ for primitive necklaces of size $5$}
        \label{tab:c_P_5}
    \end{table}
    
    \begin{table}[h!]
        \centering
        \begin{tabular}{|c|c|c|c|}
            \hline
            $P$ & $c_P$ & $|\mathcal{O}_{P^k}|$ & Verified for \\
            \hline
            $BWWWWW$ & 164 & $231\cdot 164^{k-1}$ & $k\leq 3$ \\
            \hline
            $BBBBBW$ & 164 & $627\cdot 164^{k-1}$ & $k\leq 3$ \\
            \hline
            \hline
            $BBWWWW$ & 96 & $185\cdot 96^{k-1}$ & $k\leq 3$ \\
            \hline
            $BBBBWW$ & 96 & $322\cdot 96^{k-1}$ & $k\leq 3$ \\
            \hline
            \hline
            $BBBWWW$ & 80 & $214\cdot 80^{k-1}$ & $k\leq 3$ \\
            \hline
            \hline
            $BWBWWW$ & 53 & $87\cdot 53^{k-1}$ & $k\leq 4$ \\
            \hline
            $BBBWBW$ & 53 & $133\cdot 53^{k-1}$ & $k\leq 4$ \\
            \hline
            \hline
            $WWBWBB$ & 38 & $80\cdot 38^{k-1}$ & $k\leq 4$ \\
            \hline
            \hline
            $BBWBWW$ & 30 & $65\cdot 30^{k-1}$ & $k\leq 4$ \\
            \hline
        \end{tabular}
        \caption{$|\mathcal{O}_{P^k}|$ and $c_P$ for primitive necklaces of size $6$}
        \label{tab:c_P_6}
    \end{table}
    
    \newpage
    
    \begin{table}[h!]
        \centering
        \begin{tabular}{|c|c|c|c|}
            \hline
            $P$ & $c_P$ & $|\mathcal{O}_{P^k}|$ & Verified for \\
            \hline
            $BWWWWWW$ & 578 & $1002\cdot 578^{k-1}$ & $k\leq 2$ \\
            \hline
            $BBBBBBW$ & 578 & $3010\cdot 578^{k-1}$ & $k\leq 2$ \\
            \hline
            \hline
            $BBWWWWW$ & 351 & $811\cdot 351^{k-1}$ & $k\leq 2$ \\
            \hline
            $BBBBBWW$ & 351 & $1637\cdot 351^{k-1}$ & $k\leq 2$ \\
            \hline
            \hline
            $BBBWWWW$ & 290 & $777\cdot 290^{k-1}$ & $k\leq 2$ \\
            \hline
            $BBBBWWW$ & 290 & $1114\cdot 290^{k-1}$ & $k\leq 2$ \\
            \hline
            \hline
            $BWBWWWW$ & 152 & $294\cdot 152^{k-1}$ & $k\leq 3$ \\
            \hline
            $BBBBWBW$ & 152 & $544\cdot 152^{k-1}$ & $k\leq 3$ \\
            \hline
            \hline
            $BWBBWWW$ & 94 & $336\cdot 94^{k-1}$ & $k\leq 3$ \\
            \hline
            $BBBWWBW$ & 94 & $286\cdot 94^{k-1}$ & $k\leq 3$ \\
            \hline
            \hline
            $BBWBWWW$ & 81 & $189\cdot 81^{k-1}$ & $k\leq 3$ \\
            \hline
            $BBBWBWW$ & 81 & $255\cdot 81^{k-1}$ & $k\leq 3$ \\
            \hline
            \hline
            $BWWBWWW$ & 75 & $150\cdot 75^{k-1}$ & $k\leq 3$ \\
            \hline
            $BBBWBBW$ & 75 & $255\cdot 75^{k-1}$ & $k\leq 3$ \\
            \hline
            \hline
            $WBWBWBW$ & 63 & $148\cdot 63^{k-1}$ & $k\leq 3$ \\
            \hline
            $BWBWBWB$ & 63 & $158\cdot 63^{k-1}$ & $k\leq 3$ \\
            \hline
            \hline
            $BBWWBWW$ & 50 & $125\cdot 50^{k-1}$ & $k\leq 4$ \\
            \hline
            $BBWBBWW$ & 50 & $145\cdot 50^{k-1}$ & $k\leq 4$ \\
            \hline
        \end{tabular}
        \caption{$|\mathcal{O}_{P^k}|$ and $c_P$ for primitive necklaces of size $7$}
        \label{tab:c_P_7}
    \end{table}
    
    \newpage
    
    \begin{table}[h!]
        \centering
        \begin{tabular}{|c|c|c|c|}
            \hline
            $P$ & $c_P$ & $|\mathcal{O}_{P^k}|$ & Verified for \\
            \hline
            $BWWWWWWW$ & 2313 & $4565\cdot 2313^{k-1}$ & $k\leq 2$ \\
            \hline
            $BBBBBBBW$ & 2313 & $14883\cdot 2313^{k-1}$ & $k\leq 2$ \\
            \hline
            \hline
            $BBWWWWWW$ & 1426 & $3727\cdot 1426^{k-1}$ & $k\leq 2$ \\
            \hline
            $BBBBBBWW$ & 1426 & $8463\cdot 1426^{k-1}$ & $k\leq 2$ \\
            \hline
            \hline
            $BBBWWWWW$ & 1185 & $3880\cdot 1185^{k-1}$ & $k\leq 2$ \\
            \hline
            $BBBBBWWW$ & 1185 & $5972\cdot 1185^{k-1}$ & $k\leq 2$ \\
            \hline
            \hline
            $BBBBWWWW$ & 956 & $4420\cdot 956^{k-1}$ & $k\leq 2$ \\
            \hline
            \hline
            $BWBWWWWW$ & 562 & $1152\cdot 562^{k-1}$ & $k\leq 2$ \\
            \hline
            $BBBBBWBW$ & 562 & $2414\cdot 562^{k-1}$ & $k\leq 2$ \\
            \hline
            \hline
            $WWWBWBBB$ & 436 & $1076\cdot 436^{k-1}$ & $k\leq 2$ \\
            \hline
            \hline
            $BBWBWWWW$ & 288 & $747\cdot 288^{k-1}$ & $k\leq 2$ \\
            \hline
            $BBBBWBWW$ & 288 & $1158\cdot 288^{k-1}$ & $k\leq 2$ \\
            \hline
            \hline
            $BWBBWWWW$ & 273 & $815\cdot 273^{k-1}$ & $k\leq 2$ \\
            \hline
            $BBBBWWBW$ & 273 & $1082\cdot 273^{k-1}$ & $k\leq 2$ \\
            \hline
            \hline
            $BBBWBWWW$ & 240 & $802\cdot 240^{k-1}$ & $k\leq 2$ \\
            \hline
            \hline
            $BWWBWWWW$ & 220 & $500\cdot 220^{k-1}$ & $k\leq 2$ \\
            \hline
            $BBBBWBBW$ & 220 & $983\cdot 220^{k-1}$ & $k\leq 2$ \\
            \hline
            \hline
            $BWBWBWWW$ & 197 & $420\cdot 197^{k-1}$ & $k\leq 3$ \\
            \hline
            $BBBWBWBW$ & 197 & $593\cdot 197^{k-1}$ & $k\leq 3$ \\
            \hline
            \hline
            $WWWBWWBB$ & 150 & $375\cdot 150^{k-1}$ & $k\leq 3$ \\
            \hline
            $WWBBWBBB$ & 150 & $525\cdot 150^{k-1}$ & $k\leq 3$ \\
            \hline
            \hline
            $WWWBBWBB$ & 135 & $414\cdot 135^{k-1}$ & $k\leq 3$ \\
            \hline
            $WWBWWBBB$ & 135 & $470\cdot 135^{k-1}$ & $k\leq 3$ \\
            \hline
            \hline
            $WWWBBWWB$ & 135 & $360\cdot 135^{k-1}$ & $k\leq 3$ \\
            \hline
            $WBBWWBBB$ & 135 & $524\cdot 135^{k-1}$ & $k\leq 3$ \\
            \hline
            \hline
            $BBWBWWBW$ & 114 & $316\cdot 114^{k-1}$ & $k\leq 3$ \\
            \hline
            \hline
            $BBWBWBWW$ & 110 & $295\cdot 110^{k-1}$ & $k\leq 3$ \\
            \hline
            \hline
            $WWBWBWBB$ & 97 & $309\cdot 97^{k-1}$ & $k\leq 3$ \\
            \hline
            \hline
            $BWBWWBWW$ & 85 & $245\cdot 85^{k-1}$ & $k\leq 3$ \\
            \hline
            $BBWBBWBW$ & 85 & $289\cdot 85^{k-1}$ & $k\leq 3$ \\
            \hline
        \end{tabular}
        \caption{$|\mathcal{O}_{P^k}|$ and $c_P$ for primitive necklaces of size $8$}
        \label{tab:c_P_8}
    \end{table}

    \newpage

\subsection*{Data on the generating functions $H_P(x)= \displaystyle \lim_{\ell\rightarrow\infty} \mathcal{D}_{P^\ell}(x)$}

    Here are $H_P(x), H_{P^*}(x)$ for some primitive necklaces $P$ and their duals $P^*$.

    \centering

    \begin{align*}
        H_{BWW} &= (1-x)\dfrac{x^3 - 3x^2 - 4x -3}{2x^3 + x^2 - 1}\\
        &= H_{BBW}\\
        \cline{1-2}\\
        H_{BWWW} &= (1-x)\dfrac{x^5 + 8x^4 - 3x^3 - 8x^2 - 6x - 4}{6x^4 + 4x^3 + x^2 - 1}\\\\
        H_{BBBW} &= (1-x)\dfrac{2x^5 + 8x^4 - 5x^3 - 10x^2 - 7x - 4}{6x^4 + 4x^3 + x^2 - 1}\\\\
        \cline{1-2}\\
        H_{BBWW} &= (1-x)\dfrac{x^5 + 4x^4 - 3x^3 - 6x^2 - 6x - 4}{3x^4 + 2x^3 + x^2 - 1}\\\\
        \cline{1-2}\\
        H_{BWWWW} &= (1-x)\dfrac{2x^6 + 16x^5 - 12x^4 - 23x^3 - 16x^2 - 8x - 5}{12x^5 + 8x^4 + 2x^3 - 1}\\\\
        H_{BBBBW} &= (1-x)\dfrac{4x^6 + 16x^5 - 16x^4 - 28x^3 - 19x^2 - 9x - 5}{12x^5 + 8x^4 + 2x^3 - 1}\\\\
        \cline{1-2}\\
        H_{BBWWW} &= (1-x)\dfrac{3x^6 + 14x^5 - 10x^4 - 19x^3 - 15x^2 - 8x - 5}{9x^5 + 6x^4 + 2x^3 - 1}\\\\
        H_{BBBWW} &= (1-x)\dfrac{3x^6 + 10x^5 - 15x^4 - 24x^3 - 18x^2 - 9x - 5}{9x^5 + 6x^4 + 2x^3 - 1}\\\\
        \cline{1-2}\\
        H_{BWBWB} &= (1-x)\dfrac{x^6 + 8x^5 - 9x^4 - 16x^3 - 12x^2 - 7x - 5}{6x^5 + 4x^4 + x^3 - 1}\\
        &= H_{WBWBW}\\
        \cline{1-2}\\
        H_{BBBWWW} &= (1-x)\dfrac{6x^8 + 31x^7 + 69x^6 - 16x^5 - 57x^4 - 46x^3 - 24x^2 - 11x - 6}{27x^6 + 20x^5 + 7x^4 + x^3 - 1}\\\\
        \cline{1-2}\\
        H_{BWWWWW} &= (1-x)\dfrac{2x^8 + 13x^7 + 51x^6 - 10x^5 - 49x^4 - 40x^3 - 21x^2 - 10x - 6}{24x^6 + 20x^5 + 7x^4 + x^3 - 1}\\\\
        H_{BBBBBW} &= (1-x)\dfrac{4x^8 + 18x^7 + 52x^6 - 18x^5 - 61x^4 - 49x^3 - 25x^2 - 11x - 6}{24x^6 + 20x^5 + 7x^4 + x^3 - 1}\\\\
        \cline{1-2}\\
        H_{BBWWWW} &= (1-x)\dfrac{3x^8 + 13x^7 + 39x^6 - 17x^5 - 43x^4 - 35x^3 - 20x^2 - 10x - 6}{18x^6 + 13x^5 + 5x^4 + x^3 - 1}\\\\
        H_{BBBBWW} &= (1-x)\dfrac{6x^8 + 20x^7 + 40x^6 - 21x^5 - 51x^4 - 43x^3 - 24x^2 - 11x - 6}{18x^6 + 13x^5 + 5x^4 + x^3 - 1}\\\\
        \cline{1-2}\\
        H_{BWBWWW} &= (1-x)\dfrac{x^8 + 8x^7 + 42x^6 - x^5 - 35x^4 - 32x^3 - 18x^2 - 9x - 6}{18x^6 + 16x^5 + 6x^4 + x^3 - 1}\\\\
        H_{BBBWBW} &= (1-x)\dfrac{2x^8 + 12x^7 + 43x^6 - 7x^5 - 43x^4 - 38x^3 - 21x^2 - 10x - 6}{18x^6 + 16x^5 + 6x^4 + x^3 - 1}\\\\
        \cline{1-2}\\
        H_{BBWBWW} &= (1-x)\dfrac{2x^7 + 12x^6 - 3x^5 - 20x^4 - 23x^3 - 16x^2 - 9x - 6}{6x^6 + 7x^5 + 4x^4 + x^3 - 1}\\\\
        \cline{1-2}\\
        H_{WWBWBB} &= (1-x)\dfrac{x^7 + 8x^6 - 15x^5 - 26x^4 - 21x^3 - 13x^2 - 8x - 6}{6x^6 + 4x^5 + x^4 - 1}\\\\
        \cline{1-2}\\
        H_{BBWBWWW} &= (1-x)\dfrac{6x^8 + 40x^7 - 9x^6 - 67x^5 - 69x^4 - 42x^3 - 21x^2 - 11x - 7}{18x^7 + 21x^6 + 10x^5 + 2x^4 - 1}\\\\
        H_{BBBWBWW} &= (1-x)\dfrac{6x^8 + 32x^7 - 23x^6 - 79x^5 - 76x^4 - 47x^3 - 24x^2 - 12x - 7}{18x^7 + 21x^6 + 10x^5 + 2x^4 - 1}\\\\
        \cline{1-2}\\
        H_{BWBWBWB} &= (1-x)\dfrac{x^9 + 8x^8 + 42x^7 - 19x^6 - 63x^5 - 56x^4 - 34x^3 - 18x^2 - 10x - 7}{18x^7 + 16x^6 + 6x^5 + x^4 - 1}\\
        &= H_{WBWBWBW}\\
        \cline{1-2}\\
        H_{BWBBWWW} &= (1-x)\dfrac{x^9 + 8x^8 + 42x^7 - 19x^6 - 63x^5 - 56x^4 - 34x^3 - 18x^2 - 10x - 7}{18x^7 + 16x^6 + 6x^5 + x^4 - 1}\\\\
        H_{BBBWWBW} &= (1-x)\dfrac{2x^9 + 12x^8 + 43x^7 - 25x^6 - 77x^5 - 70x^4 - 43x^3 - 22x^2 - 11x - 7}{18x^7 + 16x^6 + 6x^5 + x^4 - 1}\\\\
        \cline{1-2}\\
        H_{BBWWBWW} &= (1-x)\dfrac{2x^8 + 12x^7 - 9x^6 - 33x^5 - 38x^4 - 28x^3 - 17x^2 - 10x - 7}{6x^7 + 7x^6 + 4x^5 + x^4 - 1}\\\\
        H_{BBWBBWW} &= (1-x)\dfrac{2x^8 + 12x^7 - 9x^6 - 32x^5 - 36x^4 - 26x^3 - 16x^2 - 10x - 7}{6x^7 + 7x^6 + 4x^5 + x^4 - 1}\\\\
        \cline{1-2}\\
        H_{BWWBWWW} &= (1-x)\dfrac{2x^8 + 23x^7 - 12x^6 - 54x^5 - 53x^4 - 33x^3 - 18x^2 - 10x - 7}{12x^7 + 14x^6 + 6x^5 + x^4 - 1}\\\\
        H_{BBBWBBW} &= (1-x)\dfrac{4x^8 + 24x^7 - 16x^6 - 60x^5 - 57x^4 - 36x^3 - 20x^2 - 11x - 7}{12x^7 + 14x^6 + 6x^5 + x^4 - 1}\\\\
    \end{align*}

\end{document}